\documentclass[a4paper,10pt]{scrartcl}

\usepackage[utf8]{inputenc}
\usepackage{amsmath,amsthm,amssymb}
\usepackage{dsfont}
\usepackage{xcolor}
\definecolor{darkblue}{rgb}{0, 0, 0.5}
\usepackage[colorlinks,linkcolor=darkblue,citecolor=darkblue,urlcolor=darkblue]{hyperref}
\usepackage{bm}
\usepackage{tikz}  % the TikZ package
\usetikzlibrary{patterns} % for crosshatch etc.
\usepackage{pgfplots} % package for plots in TikZ
\pgfplotsset{compat=newest} % use newest version

\theoremstyle{plain}
\newtheorem{theorem}{Theorem}
\newtheorem{lemma}{Lemma}

\theoremstyle{definition}
\newtheorem{assumption}{Assumption}
\newtheorem{definition}{Definition}

\theoremstyle{remark}
\newtheorem{remark}{Remark}

\usepackage{bm}
\newcommand{\yb}{\bm{y}}
\newcommand{\Ab}{\bm{A}}
\newcommand{\Bb}{\bm{B}}
\newcommand{\Cb}{\bm{C}}
\newcommand{\Mb}{\bm{M}}
\newcommand{\Pb}{\bm{P}}
\newcommand{\Qb}{\bm{Q}}
\newcommand{\Rb}{\bm{R}}
\newcommand{\Sb}{\bm{S}}
\newcommand{\Wb}{\bm{W}}

\newcommand{\R}{\mathbb{R}}

\newcommand{\ordo}{\mathcal{O}}
\newcommand{\I}{\mathcal{I}}

\newcommand{\iprod}[2]{\langle #1, #2 \rangle}
\renewcommand{\exp}[1]{\mathrm{e}^{#1}}

\newcommand{\ds}{\mathrm{d}s}
\newcommand{\dt}{\mathrm{d}t}

\DeclareOldFontCommand{\sc}{\normalfont\scshape}{\@nomath\sc}
\DeclareOldFontCommand{\bf}{\normalfont\bfseries}{\mathbf}

\newcommand{\email}[1]{Email: \href{mailto:#1}{\nolinkurl{#1}}}

\begin{document}

\title{A linear implicit Euler method for the finite element discretization of a controlled stochastic heat equation}

\author{Peter Benner\thanks{Max Planck Institute for Dynamics of Complex Technical Systems, Sandtorstrasse~1, 39106 Magdeburg, Germany.
         \email{benner@mpi-magdeburg.mpg.de}}
\and
Tony Stillfjord\thanks{Centre for Mathematical Sciences, Lund University,
           Box 118, 221 00 Lund, Sweden.\vskip0pt
           \email{tony.stillfjord@math.lth.se}}
\and
Christoph Trautwein\thanks{Institute  for  Mathematics,  Friedrich  Schiller  University Jena,
           Ernst–Abbe–Platz 2, 07743 Jena,  Germany,
           \email{christoph.trautwein@uni-jena.de}}
}

\date{}
\maketitle

\renewcommand*{\raggedsection}{\centering}
\subsubsection*{Abstract}
We consider a numerical approximation of a linear quadratic control problem constrained by the stochastic heat equation with non-homogeneous Neumann boundary conditions. This involves a combination of distributed and boundary control, as well as both distributed and boundary noise.
We apply the finite element method for the spatial discretization and the linear implicit Euler method for the temporal discretization.
Due to the low regularity induced by the boundary noise, convergence orders above $1/2$ in space and $1/4$ in time cannot be expected. We prove such optimal convergence orders for our full discretization when the distributed noise and the initial condition are sufficiently smooth. Under less smooth conditions, the convergence order is further decreased.
 Our results only assume that the related (deterministic) differential Riccati equation can be approximated with a certain convergence order, which is easy to achieve in practice. 
We confirm these theoretical results through a numerical experiment in a two dimensional domain.
% % Keywords:
% {optimal control; stochastic partial differential equation; full discretization; convergence analysis; heat conduction.}

\renewcommand*{\raggedsection}{\flushleft}

\section{Introduction}

This paper is devoted to a numerical scheme for a linear quadratic control problem constrained by the stochastic heat equation with non-homogeneous Neumann boundary conditions.
We prove optimal convergence orders for a full discretization which combines a linear implicit Euler method in time and a finite element discretization in space.

For time-dependent heat distributions considered in a bounded domain, noise terms in the sense of random heating or cooling phenomenas arise due to imperfect insulation and other uncertain environmental effects.
In engineering applications, this might lead to undesired behavior. To keep a desired heat profile, it is therefore necessary to regulate the system. This task can be formulated as a linear quadratic control problem constrained by the stochastic heat equation, where controls and additive noise terms are located inside the domain as well as on the boundary.
Here, we treat the case of noise terms defined by Q-Wiener processes.
In stochastic control theory, it is well known that the concept of mild solutions is useful to include non-homogeneous boundary conditions, see \cite{Debussche07,Fabbri09,Fabbri17,Guatteri11,Yu11}.
In this context, we also refer to related deterministic control problems, see \cite{BennerMena18,Bensoussan07} and the references therein.
Typically, optimal controls as solutions of stochastic linear quadratic control problems are characterized by a feedback law, see \cite{Ahmed81,Benner18,Curtain78,Duncan12,Hu18}.
These feedback laws often involve the solution to a suitable operator-valued differential Riccati equation.
In this paper, the Riccati equation is deterministic resulting from the fact that only additive noise terms are included.
As a consequence, the optimal heat distribution fulfills a system of a linear stochastic partial differential equation (SPDE), referred to as the controlled stochastic heat equation, which is coupled to the operator-valued differential Riccati equation.
The main obstacle is that both the controlled stochastic heat equation as well as the Riccati equation can not be solved explicitly.
For that reason, we analyze a numerical approximation of the system describing the optimal heat distribution.

For the spatial discretization, we use the finite element method as introduced in \cite{Thomee06}, where only parabolic equations with homogeneous boundary conditions are considered.
The case of non-homogeneous boundary conditions is studied in, e.g.,~\cite{Lasiecka86}.
Here, we need a generalization of this theory since the system includes Q-Wiener processes.
Numerical simulations for Q-Wiener processes with values in Hilbert spaces as well as for some specific SPDEs are demonstrated in \cite{LordPowellShardlow}.

Temporal discretization of SPDEs has become an active research area within the last years.
Equations driven by additive noise terms are considered in \cite{Wang17}, and \cite{Kruse14,Lord13,Tambue19} also consider the case of multiplicative noise terms.
These papers have in common that the linear implicit Euler method is used for the temporal discretization. This is essentially the usual implicit Euler method but with the noise terms treated explicitly, since treating them implicitly makes no sense. The stochastic part of the equation is therefore treated explicitly, and the deterministic part implicitly. We follow the same approach in this paper.
The error analyses are mainly based on the fact that the underlying equation involves a closed operator generating an analytic semigroup, such that fractional powers of this closed operator are well defined.

The shortcoming of the papers mentioned above is that they only consider equations with homogeneous boundary conditions
We will extend these results by including instead non-homogeneous Neumann boundary conditions. Because this leads to a less regular solution, the convergence order is decreased. 
However, the theory of fractional powers to closed operators can still be applied, and we use this to prove optimal convergence orders under the assumption that the associated Riccati equation can be well approximated. We make such an assumption mainly because there is a lack of temporal error analyses applicable to the current situation, and providing such a proof is out of the scope of this paper. We refer to \cite{Lasiecka00} for related results on deterministic linear quadratic control problems and their corresponding Riccati equations.

In order to illustrate our theoretical results, we implement our method in MATLAB and perform a numerical experiment which shows the expected convergence orders. We also confirm that achieving the assumed convergence orders for the approximation of the Riccati equation is straightforward in practice.

The paper is organized as follows.
In Section \ref{sec:controlproblem}, we introduce the linear quadratic control problem constrained by the stochastic heat equation.
We state the optimal controls and derive the resulting system describing the optimal heat distribution.
Section \ref{sec:discretization} is devoted to the numerical scheme of the controlled stochastic heat equation and the Riccati equation. We also state the main result concerning the convergence order. In order to prepare for the proof of this theorem, we derive several auxiliary results on continuity, consistency and stability in Section~\ref{sec:auxiliary}. The proof of the main result then follows in Section \ref{sec:proof}.
Finally, in Section \ref{sec:numericalexperiments}, we discuss the implementation and illustrate the theoretical results through a numerical experiment.

\section{A linear quadratic control problem constrained by the stochastic heat equation}\label{sec:controlproblem}

Throughout this paper, let $(\Omega,\mathcal{F}, \mathbb{P})$ be a complete probability space endowed with a filtration $(\mathcal{F}_t)_{t \geq 0}$ satisfying $\mathcal F_t = \bigcap_{s > t} \mathcal F_s$ for all $t \geq 0$ and $\mathcal F_0$ contains all sets of $\mathcal F$ with $\mathbb P$-measure 0. We use $\mathbb E$ to denote the expectation with respect to this probability space.
Moreover, we assume that $\mathcal D \subset \mathbb{R}^n$ for $n \geq 1$ is either a bounded domain with sufficiently smooth boundary $\partial \mathcal D$ or a bounded and convex domain.

First, we introduce some basic notations and we state properties of operators frequently used in the remaining part.
For $s \geq 0$, let $H^s(\mathcal D)$ denote the usual Sobolev space.
We set $H = L^2(\mathcal D)$ and let $I$ denote the identity operator on $H$. We introduce the Neumann realization of the Laplace operator $A\colon D(A) \subset H \rightarrow H$ defined by
\begin{equation*}
 A y = \Delta y
\end{equation*}
for every $y \in D(A)$ with 
\begin{equation*}
 D(A) = \left\{ y \in H^2(\mathcal D) : \frac{\partial}{\partial \nu} y = 0 \text{ on } \partial \mathcal D \right\}.
\end{equation*}
The characterization of the domain results from existence and uniqueness results of the corresponding elliptic problem, see \cite{Grisvard85}.
The operator $A$ is the infinitesimal generator of an analytic semigroup $\left(e^{A t}\right)_{t \geq 0}$ of contractions such that for $\lambda>0$, fractional powers of $\lambda-A$ denoted by $(\lambda-A)^\alpha$ with $\alpha \in \mathbb R$ are well defined.
For more details in a more general framework, we refer to \cite{Pazy83,Vrabie03}, but we have also collected the main properties which we need in Section~\ref{sec:auxiliary}.

For $\alpha \in \mathbb R$, the space $D((\lambda-A)^\alpha)$ equipped with the inner product 
\begin{equation*}
 \langle y,z \rangle_\alpha = \left\langle (\lambda-A)^\alpha y,(\lambda-A)^\alpha z \right\rangle_H
\end{equation*}
becomes a Hilbert space.
The corresponding norm is denoted by $\|\cdot\|_\alpha$.
In general, the domain of $(\lambda-A)^\alpha$ for $\alpha \in (0,1)$ can be expressed explicitly by interpolation of the spaces $H$ and $D(A)$, see \cite{Lions72}.
In case that $\mathcal D$ is bounded with sufficiently smooth boundary, we have
\begin{equation*}
 D((\lambda-A)^\alpha) = 
 \begin{cases} 
  H^{2 \alpha}(\mathcal D) &\mbox{for } \alpha \in \left(0,3/4\right), \\
  \left\{ y \in H^{2 \alpha}(\mathcal D) : \frac{\partial}{\partial \nu} y = 0 \text{ on } \partial \mathcal D \right\} & \mbox{for } \alpha \in \left(3/4,1\right),
 \end{cases}
\end{equation*}
where we refer to \cite{Fujiwara67}.
We set $H_b = L^2(\partial \mathcal D)$ and introduce the Neumann operator $N \colon H_b \rightarrow H$ given by $g = N h$ with
\begin{equation}\label{neumann}
 \left\{
 \begin{aligned}
  \Delta g(x) &= \lambda \, g(x) \quad \text{in } \mathcal D, \\
  \frac{\partial}{\partial \nu} g(x) &= h(x) \hspace*{0.6cm} \text{on } \partial \mathcal D,
 \end{aligned}
 \right.
\end{equation}
where $\lambda > 0$.
If $\mathcal D$ is bounded with sufficiently smooth boundary, the result $N \in \mathcal{L}\left( H_b; H^{3/2}(\mathcal D)\right)$ was proven in \cite{Lions72}.
In this case, we can therefore conclude that $N \in \mathcal{L}(H_b; D((\lambda-A)^\alpha))$ for $\alpha \in \left(0,3/4\right)$, which means that the operator $(\lambda-A)^\alpha N$ is linear and bounded by the closed graph theorem.
If $\mathcal D$ is instead bounded and convex, then $\mathcal D$ has a Lipschitz boundary and satisfies the cone property, see \cite{Grisvard85}. 
We therefore again obtain $N \in \mathcal{L}(H_b; D((\lambda-A)^\alpha))$ for $\alpha \in \left(0,3/4\right)$, see \cite{Lasiecka80}.

Next, we introduce the controlled stochastic heat equation with non-homogeneous Neumann boundary conditions as an evolution equation.
Here, we include distributed and boundary controls as well as distributed and boundary noise.
Let $U$ contain all $\mathcal{F}_t$-adapted processes $(u(t))_{t \in [0,T]}$ with values in an arbitrary Hilbert space $\bar U$ satisfying $\mathbb E \int_0^T \|u(t)\|_{\bar U}^2 \, dt < \infty$ and let $V$ contain all $\mathcal{F}_t$-adapted processes $(v(t))_{t \in [0,T]}$ with values in $\bar V \subset H_b$ satisfying $\mathbb E \int_0^T \|v(t)\|_{H_b}^2 \, dt < \infty$.
We consider the following controlled system in $H$ for $t \in [0,T]$ and $\lambda > 0$:
\begin{equation}\label{stochheat}
\left\{
 \begin{aligned}
  \mathrm{d} y(t) &= \left[ A y(t) + B u(t) + (\lambda-A) N v(t)\right] \dt + G \, \mathrm{d}W(t) + (\lambda-A) N \, \mathrm{d}W_b(t), \\
  y(0) &= \xi,
 \end{aligned}
\right.
\end{equation}
where $(u(t))_{t \in [0,T]}$ and $(v(t))_{t \in [0,T]}$ represent the distributed and the boundary controls.
We assume that $u \in U$, $B \in \mathcal L(\bar U;\,H)$, and $v \in V$.
The processes $(W(t))_{t \geq 0}$ and $(W_b(t))_{t \geq 0}$ are independent and $\mathcal{F}_t$-adapted Q-Wiener processes with values in $H$ and $H_b$, respectively.
The corresponding covariance operators are denoted by $Q \in \mathcal L(H)$ and $Q_b \in \mathcal L(H_b)$.
We make the following assumptions:

\begin{assumption}\label{initialassumption}
 The initial value $\xi \in L^2(\Omega;\,D((\lambda-A)^{\beta/2}))$ with $\beta \in (0,2)$ is $\mathcal F_0$-measurable.
\end{assumption}

\begin{remark}
 The results shown in this section also holds for an $\mathcal F_0$-measurable initial value $\xi \in L^2(\Omega;\,H)$.
 We make the additional regularity requirement due to the main result stated in the following section.
\end{remark}

\begin{assumption}\label{noiseoperator}
 We assume that $G$ is a square integrable random variable with values in the space of Hilbert-Schmidt operators mapping $Q^{1/2}(H)$ into $D((\lambda-A)^{\beta/2})$ denoted by $\mathcal L_{HS}(Q^{1/2}(H);\,D((\lambda-A)^{\beta/2}))$, where $\beta \in (0,2)$ arises from Assumption \ref{initialassumption}.
\end{assumption}

\begin{definition}\label{stochheatdef}
 A predictable process $(y(t))_{t \in [0,T]}$ with values in $H$ is a mild solution of system \eqref{stochheat} if
 \begin{equation*}
 \sup_{t \in [0,T]} \mathbb E \|y(t)\|_H^2 < \infty
 \end{equation*}
 and for all $t \in [0,T]$ and $\mathbb P$-a.s.
 \begin{align*}
  y(t) &= e^{A t} \xi + \int\limits_0^t e^{A (t-s)} B u(s) \, \ds + \int\limits_0^t (\lambda-A) e^{A (t-s)} N v(s) \, \ds + \int\limits_0^t e^{A (t-s)} G \, \mathrm{d}W(s) \\
  &\quad + \int\limits_0^t (\lambda-A) e^{A (t-s)} N \, \mathrm{d}W_b(s).
 \end{align*}
\end{definition}

For an existence and uniqueness result of a mild solution to system \eqref{stochheat}, we refer to \cite{Benner18}.
Next, we introduce the cost functional $J \colon U \times V \rightarrow \mathbb R$ defined by
\begin{equation*}
 J(u,v) = \mathbb E \left[ \int\limits_0^T \langle C \, y(t), C \, y(t) \rangle_Z + \langle R \, u(t), u(t) \rangle_H + \langle R_b \, v(t), v(t) \rangle_{H_b} \dt \right],
\end{equation*}
where $C \in \mathcal L(H;Z)$ represents an observation operator mapping $H$ into an arbitrary Hilbert space $Z$.
The operators $R \in \mathcal L(H)$ and $R_b \in \mathcal L(H_b)$ are given scaling factors for the costs of the controls and are assumed to be invertible.
The aim is to find controls $\overline u \in U$ and $\overline v \in V$ such that
\begin{equation*}
 J(\overline u, \overline v) = \inf_{u \in U, v \in V} J(u,v).
\end{equation*}
The controls $\overline u \in U$ and $\overline v \in V$ are called optimal controls.
In \cite{Ahmed81,Benner18,Curtain78,Duncan12,Fabbri09,Hu18}, similar control problems are considered with the result that the optimal controls satisfy a feedback law.
We follow the same approach here, and therefore introduce the following Riccati equation in $\mathcal L(H)$:
\begin{equation}\label{riccati}
\left\{
 \begin{aligned}
  \frac{\mathrm{d}}{\dt}\mathcal P(t) &= A \mathcal P(t) + \mathcal P(t) A - \mathcal P(t)B R^{-1} B^*\mathcal P(t) - \mathcal{H}^*(t) \mathcal{G} R_b^{-1} \mathcal{G}^* \mathcal{H}(t) +C^*C\\
  \mathcal P(T) &= 0,
 \end{aligned}
\right.
\end{equation}
where $\mathcal{H}(t) = (\lambda-A)^{1-\alpha}\mathcal P(t)$, $\mathcal{G} = (\lambda-A)^\alpha N$ with $\alpha \in (1/2,3/4)$.
We make the following definition, where $\Sigma(H)$ denotes the space of all symmetric operators on $H$ and $C([0,T];\Sigma(H))$ is endowed with the topology of uniform convergence:

\begin{definition}\label{riccatidef}
 The process $(\mathcal P(t))_{t \in [0,T]}$ is a mild solution of \eqref{riccati} if 
 \begin{itemize}
 \item $\mathcal P \in C([0,T];\Sigma(H))$,
 \item $\mathcal{P}(t) y \in D((\lambda-A)^{1-\alpha})$ for every $y \in H$ and all $t \in [0,T)$,
 \item $(\lambda-A)^{1-\alpha}\mathcal{P} \in C([0,T);\mathcal{L}(H))$,
 \item $\lim_{t \rightarrow 0} t^{1-\alpha} (\lambda-A)^{1-\alpha}\mathcal{P}(t) y = 0$ for every $y \in H$,
\end{itemize}
 and for all $t \in [0,T]$ and every $y \in H$
 \begin{align}\label{riccatimild}
   \mathcal{P}(t) y =& - \int\limits_t^T e^{A(s-t)}\mathcal{P}(s)B R^{-1} B^* \mathcal{P}(s)e^{A(s-t)}y \, \ds \nonumber \\
   & - \int\limits_t^T e^{A(s-t)}[ \mathcal{H}^*(s) \mathcal{G} R_b^{-1}\mathcal{G}^* \mathcal{H}(s) - C^*C ] e^{A(s-t)}y \, \ds.
 \end{align}
\end{definition}

In \cite[Part IV]{Bensoussan07}, existence and uniqueness results of a mild solution to system \eqref{riccati} are shown for some special cases.
The ideas of these proofs are easily adapted to the current situation, and therefore there exists a unique mild solution of system \eqref{riccati}.

\begin{remark}\label{riccatirem}
 Equation \eqref{riccatimild} can be written equivalently as
 \begin{align*}
  \frac{\mathrm{d}}{\dt} \langle \mathcal P(t) y, z \rangle_H &= \langle \mathcal P(t) y, A z \rangle_H + \langle\mathcal P(t) A y, z \rangle_H - \langle R^{-1} B^*\mathcal P(t) y, B^*\mathcal P(t) z \rangle_H \\
  &\quad - \langle R_b^{-1} \mathcal{G}^* \mathcal{H}(t) y, \mathcal{G}^* \mathcal{H}(t) z \rangle_H + \langle C y, C z \rangle_Z
 \end{align*}
 for every $y,z \in D(A)$, see \cite{Bensoussan07}.
\end{remark}

The optimal controls $\overline u \in U$ and $\overline v \in V$ satisfy a.e. on $[0,T]$ and $\mathbb P$-a.s.
\begin{align*}
 &\overline u(t) = - R^{-1} B^* \mathcal P(t) y(t), &\overline v(t) = - R_b^{-1} \mathcal G^* \mathcal H(t) y(t).
\end{align*}
Plugging in these formulas in \eqref{stochheat} results in the following controlled system in $H$:
\begin{equation}\label{stochheatcontrolled}
\left\{
 \begin{aligned}
  \mathrm{d} y(t) &= \left[ A y(t) - B R^{-1} B^* \mathcal P(t) y(t) - (\lambda-A) N R_b^{-1} \mathcal G^* \mathcal H(t) y(t)\right] \dt \\
  &\quad + G \, \mathrm{d}W(t) + (\lambda-A) N \, \mathrm{d}W_b(t), \\
  y(0) &= \xi.
 \end{aligned}
\right.
\end{equation}

\section{A linear implicit Euler method for the finite element discretization}\label{sec:discretization}

In this section, we introduce a fully discrete scheme for system \eqref{stochheatcontrolled}.
We denote by $\mathcal T_h$ a triangulation of the domain $\mathcal D$ with meshwidth $h \in (0,1]$.
Let $Y_h \subset Y = D((\lambda-A)^{1/2})$ be the set of continuous functions that are piecewise linear over $\mathcal T_h$.
We introduce the $L^2$-projection $P_h \colon H \rightarrow Y_h$ defined by
\begin{equation*}
 \langle P_h y, z\rangle_H = \langle y, z\rangle_H
\end{equation*}
for every $y \in H$ and every $z \in Y_h$.
Then we have the following basic estimate
\begin{equation}\label{projectionrate}
 \| y - P_h y\|_H \leq K h^\rho \|y\|_{\rho/2}
\end{equation}
for a constant $K>0$ and every $y \in D((\lambda-A)^{\rho/2})$ with $\rho \in [0,2]$, see \cite{Thomee06}.
Moreover, let $R_h \colon Y \rightarrow Y_h$ be the $Y$-projection given by
\begin{equation*}
 \langle (\lambda-A) R_h y, z\rangle_H = \langle (\lambda-A) y, z\rangle_H
\end{equation*}
for every $y \in Y$ and every $z \in Y_h$.
We have the following relation between the $L^2$-projection $P_h$ and the $Y$-projection $R_h$:
\begin{equation}\label{projectionequation}
 (\lambda-A_h) R_h y = P_h (\lambda-A) y
\end{equation}
for every $y \in D(A)$, see \cite[Lemma 3.1]{Lord13}.
We consider the following semi-discrete version of system \eqref{stochheatcontrolled} in $Y_h$:
\begin{equation}\label{stochheatsemidiscr}
\left\{
 \begin{aligned}
  \mathrm{d} y_h(t) &= \left[ A_h y_h(t) - B_h R^{-1} B_h^* \mathcal P_h(t) y_h(t) - B_h^b R_b^{-1} \left( B_h^b \right)^* \mathcal P_h(t) y_h(t)\right] \dt \\
  &\quad + P_h G \, \mathrm{d}W(t) + B_h^b \, \mathrm{d}W_b(t), \\
  y^h(0) &= P_h \xi,
 \end{aligned}
 \right.
\end{equation}
where the operator $A_h \colon Y_h \rightarrow Y_h$ satisfies for every $y,z \in Y_h$
\begin{equation*}
 \langle A_h y, z\rangle_H = \langle A y, z\rangle_H
\end{equation*}
and $B_h = P_h B$.
As a consequence of inequality \eqref{projectionrate}, we get
\begin{align}\label{inputoperatorrate}
 \| B^*y - B_h^* y\|_H \leq K h^\rho \|y\|_{\rho/2}
\end{align}
for a constant $K>0$ and every $y \in D((\lambda-A)^{\rho/2})$ with $\rho \in [0,2]$.
Moreover, we have $B_h^b = (\lambda-A_h) R_h N \in \mathcal L(H_b;\,H)$ and $(\mathcal P_h(t))_{t \in [0,T]}$ with $\mathcal P_h(t) \in \mathcal L(Y_h)$ is the solution of the semi-discrete version of system \eqref{riccati} given by
\begin{equation}\label{riccatisemidiscr}
\left\{
 \begin{aligned}
 \frac{\mathrm{d}}{\dt}\mathcal P_h(t) &= A_h \mathcal P_h(t) + \mathcal P_h(t) A_h - \mathcal P_h(t)B_h R^{-1} B_h^*\mathcal P_h(t) \\
 &\quad - \mathcal P_h(t) B_h^b R_b^{-1} \left( B_h^b \right)^* \mathcal P_h(t) +C_h^*C_h,\\
 \mathcal P_h(T) &= 0,
 \end{aligned}
 \right.
\end{equation}
where $C_h = C P_h$.
By definition, the operator $A_h$ is again the infinitesimal generator of an analytic semigroup $(e^{A_h t})_{t \geq 0}$ on $Y_h$ such that fractional powers of $\lambda-A_h$ with $\lambda>0$ are well defined.
We can therefore introduce the solutions of system \eqref{stochheatsemidiscr} and \eqref{riccatisemidiscr} in a mild sense analogously to Definitions \ref{stochheatdef} and \ref{riccatidef}.
We note that the mild solution of system \eqref{riccatisemidiscr} coincides again with the weak solution according to Remark \ref{riccatirem}.

Next, let $t_0,t_1,...,t_M$ be a partition of the time interval $[0,T]$ such that $0 = t_0 < t_1 < ... < T_M = T$.
We assume that $t_m-t_{m-1} = \Delta t$ for each $m=1,...,M$ with $\Delta t \in (0,1]$.
Applying a linear implicit Euler method to system \eqref{stochheatsemidiscr} gives us the following fully discrete system in $Y_h$ for $m=1,...,M$:
\begin{equation}\label{stochheatfullydiscr}
\left\{
 \begin{aligned}
  y_h^m &= S_{h,\Delta t} y_h^{m-1} - \Delta t S_{h,\Delta t} B_h R^{-1} B_h^* \mathcal P_h^{m-1} y_h^{m-1} - \Delta t S_{h,\Delta t} B_h^b R_b^{-1} \left( B_h^b \right)^* \mathcal P_h^{m-1} y_h^{m-1} \\
  &\quad + S_{h,\Delta t} P_h G \, \delta W_m + S_{h,\Delta t} B_h^b \, \delta W_{b,m}, \\
  y_h^0 &= P_h \xi,
 \end{aligned}
 \right.
\end{equation}
where $S_{h,\Delta t} = (I - \Delta t A_h)^{-1}$, $\delta W_{m-1} = W(t_m) -W(t_{m-1})$ and $\delta W_{b,m} = W_b(t_m) -W_b(t_{m-1})$.
The operator $\mathcal P_h^m \in \mathcal L(Y_h)$ results from a time discretization of system \eqref{riccatisemidiscr}.
We make the following assumption.

\begin{assumption}\label{riccatiassumption}
 We require for each $m=0,1,...,M-1$
 \begin{gather*}
  \|\mathcal P(t_m) - \mathcal P_h^m P_h \|_{\mathcal L(H)} \leq c \, (h^2 + \Delta t), \\
  \| \mathcal G^* \mathcal H(t_m) - \left( B_h^b \right)^* \mathcal P_h^m P_h \|_{\mathcal L(H)} \leq c \, (h + \Delta t^{1/4}),
 \end{gather*}
 where $c>0$ is a constant.
\end{assumption}

\begin{remark}
 Note that we can at least write formally $\mathcal G^* \mathcal H(t) = \left( B_h^b \right)^* P(t)$ for all $t \in [0,T)$.
 Hence, the Assumption \ref{riccatiassumption} provides especially the convergence rate for the operator $\left( B_h^b \right)^* P(t_m) - \left( B_h^b \right)^* \mathcal P_h^m P_h$ for each $m=0,1,...,M-1$.
 For some convergence results, we refer to \cite{Lasiecka00}.
 Here, we will verify the convergence rates by a numerical experiment in Section \ref{sec:numericalexperiments}.
\end{remark}

\noindent We are now in a position to state the main result of the paper:

\begin{theorem}\label{mainresult}
 Let $(y(t))_{t \in [0,T]}$ be the mild solution of system \eqref{stochheatcontrolled} and let $y_h^m$ satisfy the fully discrete system \eqref{stochheatfullydiscr} for $m=0,1,...,M-1$.
 If Assumptions \ref{initialassumption}--\ref{riccatiassumption} are fulfilled, then there exists a constant $c>0$ such that for sufficiently small $\varepsilon > 0$,
 \begin{equation*}
  \|y(t_m)-y_h^m\|_{L^2(\Omega;\,H)} \leq c \left( h^{\min\{1/2-\varepsilon,\beta\}} +\Delta t^{\min\{1/4-\varepsilon,\beta/2\}} \right).
 \end{equation*}
\end{theorem}
\noindent The proof of this theorem will be provided in Section~\ref{sec:proof}. In order to prepare, we will first collect and derive a number of lemmata in the following section.

\section{Auxiliary results}\label{sec:auxiliary}
We start by collecting some well-known properties of fractional powers of operators. For a proof, see e.g.~\cite{Pazy83,Vrabie03}.

\begin{lemma}\label{fractionalprop}
 Let $A\colon D(A) \subset H \rightarrow H$ be the Neumann realization of the Laplace operator.
 Then
 \begin{enumerate}
  \item for $\alpha \leq 0$, the operator $(\lambda-A)^\alpha$ is linear and bounded and for $\alpha>0$, the operator $(\lambda-A)^\alpha$ is linear and closed;
  \item $\alpha \geq \beta \geq 0$ implies $D\left( (\lambda-A)^\alpha \right) \subset D( (\lambda-A)^\beta )$ and for every $y \in D\left( (\lambda-A)^\alpha \right)$
   \begin{equation*}
    \| (\lambda-A)^\beta y \|_H \leq M_0 \| (\lambda-A)^\alpha y \|_H;
   \end{equation*}
  \item $D\left( (\lambda-A)^\alpha \right)$ with $\alpha>0$ is dense in $H$;
  \item $(\lambda-A)^{\alpha+\beta} y = (\lambda-A)^\alpha (\lambda-A)^\beta y$ if $y \in D\left( (\lambda-A)^\gamma \right)$, where $\gamma= \max \{\alpha, \beta, \alpha+\beta\}$;
  \item for $\alpha>0$ and $t>0$, we have $e^{A t}\colon H \rightarrow D\left((\lambda-A)^\alpha\right)$ and $(\lambda-A)^\alpha e^{A t} y= e^{A t}(\lambda-A)^\alpha y$ if $y \in D\left( (\lambda-A)^\alpha \right)$;
  \item the operator $(\lambda-A)^\alpha e^{A t}$ is linear and bounded for $\alpha>0$ and $t > 0$. 
  Moreover, we have for every $y \in H$
  \begin{equation*}
   \| (\lambda-A)^\alpha e^{A t} y \|_H \leq M_\alpha t^{-\alpha} \|y \|_H;
  \end{equation*}
  \item we have for every $y \in D\left( (\lambda-A)^\alpha \right)$ with $\alpha \in (0,1]$ and all $t>0$
  \begin{equation*}
   \| e^{A t} y - y \|_H \leq c_\alpha t^\alpha \|(\lambda-A)^\alpha y \|_H.
  \end{equation*}
 \end{enumerate}
\end{lemma}

\subsection{Continuity of mild solutions to the controlled system}

Next, we show some useful properties of the exact mild solution $y$ to the controlled system~\eqref{stochheatcontrolled}.
In the following, we use $c > 0$ as a generic constant, which may take different values at different points.

\begin{lemma}\label{mildsolbound}
 Let $(y(t))_{t \in [0,T]}$ be the mild solution of system \eqref{stochheatcontrolled}.
 If Assumptions \ref{initialassumption} and \ref{noiseoperator} hold, then there exists a constant $c > 0$ such that for all $t \in [0,T]$,
 \begin{equation*}
  \|y(t)\|_{L^2(\Omega;\,H)} \leq c \left(1 + \|\xi\|_{L^2(\Omega;\,D((\lambda-A)^{\beta/2})}\right).
 \end{equation*}
\end{lemma}

\begin{proof}
 By definition, we get
 \begin{equation}\label{inequality1}
  \|y(t)\|_{L^2(\Omega;\,H)} \leq \I_1(t) + \I_2(t) + \I_3(t),
 \end{equation}
 where
 \begin{align*}
  \I_1(t) &= \left\| e^{A t} \xi + \int\limits_0^t e^{A (t-s)} B R^{-1} B^* \mathcal P(s) y(s) \, \ds \right\|_{L^2(\Omega;\,H)}, \\
  \I_2(t) &= \left\| \int\limits_0^t (\lambda-A) e^{A (t-s)} N R_b^{-1} \mathcal G^* \mathcal H(s) y(s) \, \ds \right\|_{L^2(\Omega;\,H)} \qquad \text{and}
 \end{align*}
 \begin{equation*}
  \I_3(t) = \left\| \int\limits_0^t e^{A (t-s)} G \, \mathrm{d}W(s) + \int\limits_0^t (\lambda-A) e^{A (t-s)} N \, \mathrm{d}W_b(s) \right\|_{L^2(\Omega;\,H)}.
 \end{equation*}
 Recall that the semigroup $(e^{A t})_{t \geq 0}$ is a contraction and that the operators $B$, $R^{-1}$ and $\mathcal P(t)$ are linear and bounded.
 Using Lemma \ref{fractionalprop} (i) and (iv) -- (vi), we have for all $\alpha \in (0,3/4)$
 \begin{align}\label{inequality2}
  \I_1(t) &\leq \left\| (\lambda-A)^{-\beta/2} e^{A t} (\lambda-A)^{\beta/2} \xi \right\|_{L^2(\Omega;\,H)} \nonumber \\
  &\quad + \int\limits_0^t \left\| (\lambda-A)^{1-\alpha} e^{A (t-s)} (\lambda-A)^{\alpha-1} B R^{-1} B^* \mathcal P(s) y(s) \right\|_{L^2(\Omega;\,H)} \ds \nonumber \\
  &\leq c \left[\left\|\xi \right\|_{L^2(\Omega;\,D((\lambda-A)^{\beta/2})} + \int\limits_0^t (t-s)^{\alpha-1} \left\| y(s) \right\|_{L^2(\Omega;\,H)} \ds \right].
 \end{align}
 Recall that the operators $(\lambda-A)^\alpha N, R_b^{-1}$, $\mathcal G^*$ and $\mathcal H(t)$ with $\alpha \in (0,3/4)$ are linear and bounded.
 By Lemma \ref{fractionalprop} (iv) -- (vi), we obtain for all $\alpha \in (0,3/4)$ that
 \begin{align}\label{inequality3}
  \I_2(t) &\leq \int\limits_0^t \left\| (\lambda-A)^{1-\alpha} e^{A (t-s)} (\lambda-A)^\alpha N R_b^{-1} \mathcal G^* \mathcal H(s) y(s) \right\|_{L^2(\Omega;\,H)} \ds \nonumber \\
  &\leq c \int\limits_0^t (t-s)^{\alpha-1} \left\| y(s) \right\|_{L^2(\Omega;\,H)} \ds. 
 \end{align}
 Due to the It\^o isometry and Lemma \ref{fractionalprop} (i) and (iv) -- (vi), we get for all $\alpha \in (1/2,3/4)$ that
 \begin{align}\label{inequality4}
  \I_3(t) &= \left( \mathbb E \int\limits_0^t \left\| e^{A (t-s)} (\lambda-A)^{-(\beta-1)/2} (\lambda-A)^{(\beta-1)/2} G \right\|_{\mathcal L_{HS}(Q^{1/2}(H);\,H)}^2 \ds\right)^{1/2} \nonumber \\
  &\quad + \left( \int\limits_0^t \left\| (\lambda-A)^{1-\alpha} e^{A (t-s)} (\lambda-A)^\alpha N \right\|_{\mathcal L_{HS}(Q_b^{1/2}(H_b);\,H)}^2 \ds \right)^{1/2} \nonumber \\
  &\leq T^{1/2} \left( \mathbb E \left\| G \right\|_{\mathcal L_{HS}(Q^{1/2}(H);\,D((\lambda-A)^{(\beta-1)/2}))}^2 \right)^{1/2} + c \, T^{2\alpha-1}.
 \end{align}
 Substituting the inequalities \eqref{inequality2}--\eqref{inequality4} in \eqref{inequality1} and applying a generalized Gr\"{o}nwall inequality, see \cite[Corollary 2]{Ye2007}, yields for all $\alpha \in (1/2,3/4)$
 \begin{align*}
  \|y(t)\|_{L^2(\Omega;\,H)} &\leq c \left\| \xi \right\|_{L^2(\Omega;\,D((\lambda-A)^{\beta/2})} + T^{1/2} \left( \mathbb E \left\| G \right\|_{\mathcal L_{HS}(Q^{1/2}(H);\,D((\lambda-A)^{(\beta-1)/2}))} \right)^{1/2} \\
  &\quad + c \, T^{2\alpha-1} + c \int\limits_0^t (t-s)^{\alpha-1} \left\| y(s) \right\|_{L^2(\Omega;\,H)} \ds \\
  &\leq c \left(1 + \left\|\xi \right\|_{L^2(\Omega;\,D((\lambda-A)^{\beta/2})}\right),
 \end{align*}
 which completes the proof.
\end{proof}

\begin{lemma}\label{mildsolcontinuity}
 Let $(y(t))_{t \in [0,T]}$ be the mild solution of system \eqref{stochheatcontrolled}.
 If Assumptions \ref{initialassumption} and \ref{noiseoperator} hold, then there exists a constant $c > 0$ such that for all $\tau_1,\tau_2 \in [0,T]$ with $\tau_1 < \tau_2$ and all $\gamma \in (0,1/4)$ with $\gamma \leq \beta/2$,
 \begin{equation*}
  \|y(\tau_2)-y(\tau_1)\|_{L^2(\Omega;\,H)} \leq c \, (\tau_2-\tau_1)^\gamma \left(1 + \|\xi\|_{L^2(\Omega;\,D((\lambda-A)^{\beta/2})}\right).
 \end{equation*}
\end{lemma}

\begin{proof}
 By definition, we get
 \begin{equation}\label{inequal1}
  \|y(\tau_2)-y(\tau_1)\|_{L^2(\Omega;\,H)} \leq \I_1 + \I_2 + \I_3 + \I_4 + \I_5,
 \end{equation}
 where
 \begin{align*}
  \I_1 &= \left\| \left[e^{A \tau_2}-e^{A \tau_1}\right] \xi \right\|_{L^2(\Omega;\,H)}, \\
  \I_2 &= \left\| \int\limits_0^{\tau_1} \left[ e^{A (\tau_2-s)} - e^{A (\tau_1-s)} \right] B R^{-1} B^* \mathcal P(s) y(s) \, \ds + \int\limits_{\tau_1}^{\tau_2} e^{A (\tau_2-s)} B R^{-1} B^* \mathcal P(s) y(s) \, \ds \right\|_{L^2(\Omega;\,H)}, \\
  \I_3 &= \left\| \int\limits_0^{\tau_1} (\lambda-A) \left[ e^{A (\tau_2-s)} - e^{A (\tau_1-s)} \right] N R_b^{-1} \mathcal G^* \mathcal H(s) y(s) \, \ds \right. \\
  &\hspace*{0.6cm} \left. + \int\limits_{\tau_1}^{\tau_2} (\lambda-A) e^{A (\tau_2-s)} N R_b^{-1} \mathcal G^* \mathcal H(s) y(s) \, \ds \right\|_{L^2(\Omega;\,H)}, \\
  \I_4 &= \left\| \int\limits_0^{\tau_1} \left[ e^{A (\tau_2-s)} - e^{A (\tau_1-s)} \right] G \, \mathrm{d}W(s) + \int\limits_{\tau_1}^{\tau_2} e^{A (\tau_2-s)} G \, \mathrm{d}W(s) \right\|_{L^2(\Omega;\,H)}, \\
  \I_5 &= \left\| \int\limits_0^{\tau_1} (\lambda-A) \left[ e^{A (\tau_2-s)} - e^{A (\tau_1-s)} \right] N \, \mathrm{d}W_b(s) + \int\limits_{\tau_1}^{\tau_2} (\lambda-A) e^{A (\tau_2-s)} N \, \mathrm{d}W_b(s) \right\|_{L^2(\Omega;\,H)}.
 \end{align*}
 Recall that the semigroup $(e^{A t})_{t \geq 0}$ is a contraction. By Lemma \ref{fractionalprop} (ii) and (vii), we obtain
 \begin{align}\label{inequal2}
  \I_1 &= \left\| \left[e^{A (\tau_2-\tau_1)}-I\right] e^{A \tau_1} \xi \right\|_{L^2(\Omega;\,H)} \nonumber \\
  &\leq c_\gamma (\tau_2-\tau_1)^\gamma \left\| e^{A \tau_1} (\lambda-A)^\gamma \xi \right\|_{L^2(\Omega;\,H)} \nonumber \\
  &\leq c_\gamma (\tau_2-\tau_1)^\gamma \left\| \xi \right\|_{L^2(\Omega;\,D((\lambda-A)^{\beta/2}))}.
 \end{align}
 Recall that the operators $B, R^{-1}, \mathcal P(t)$ are linear and bounded.
 Using Lemma \ref{fractionalprop} (i) and (iv) -- (vii) and Lemma \ref{mildsolbound}, we have that
 \begin{align}\label{inequal3}
  \I_2 &\leq \int\limits_0^{\tau_1} \left\| \left[ e^{A (\tau_2-\tau_1)} - I \right] e^{A (\tau_1-s)}  B R^{-1} B^* \mathcal P(s) y(s) \right\|_{L^2(\Omega;\,H)} \ds \nonumber \\
  &\quad + \int\limits_{\tau_1}^{\tau_2} \left\| (\lambda-A)^{1-\gamma} e^{A (\tau_2-s)} (\lambda-A)^{\gamma-1} B R^{-1} B^* \mathcal P(s) y(s) \right\|_{L^2(\Omega;\,H)} \ds \nonumber \\
  &\leq c \left[ (\tau_2-\tau_1)^\gamma \int\limits_0^{\tau_1} (\tau_1-s)^{-\gamma} \left\| y(s) \right\|_{L^2(\Omega;\,H)} \ds + \int\limits_{\tau_1}^{\tau_2} (\tau_2-s)^{\gamma-1} \left\| y(s) \right\|_{L^2(\Omega;\,H)} \ds \right] \nonumber \\
  &\leq c \, (\tau_2-\tau_1)^\gamma \left(1 + \left\|\xi \right\|_{L^2(\Omega;\,D((\lambda-A)^{\beta/2})}\right).
 \end{align}
 Recall that the operators $(\lambda-A)^\alpha N, R_b^{-1}, \mathcal G^*, \mathcal H(t)$ with $\alpha \in (0,3/4)$ are linear and bounded.
 Lemma \ref{fractionalprop} (iv) -- (vii) and Lemma \ref{mildsolbound} give us for all $\alpha \in (\gamma,3/4)$
 \begin{align}\label{inequal4}
  \I_3 &\leq  \int\limits_0^{\tau_1} \left\| \left[ e^{A (\tau_2-\tau_1)} - I \right] (\lambda-A)^{1-\alpha} e^{A (\tau_1-s)} (\lambda-A)^\alpha N R_b^{-1} \mathcal G^* \mathcal H(s) y(s) \right\|_{L^2(\Omega;\,H)} \ds \nonumber \\
  &\quad + \int\limits_{\tau_1}^{\tau_2} \left\| (\lambda-A)^{1-\gamma} e^{A (\tau_2-s)} (\lambda-A)^\gamma N R_b^{-1} \mathcal G^* \mathcal H(s) y(s) \right\|_{L^2(\Omega;\,H)} \ds \nonumber \\
  &\leq c \, (\tau_2-\tau_1)^\gamma \int\limits_0^{\tau_1} (\tau_1-s)^{\alpha-\gamma-1} \left\| y(s) \right\|_{L^2(\Omega;\,H)} \ds + c \int\limits_{\tau_1}^{\tau_2} (\tau_2-s)^{\gamma-1} \left\| y(s) \right\|_{L^2(\Omega;\,H)} \ds \nonumber \\
  &\leq c \, (\tau_2-\tau_1)^\gamma \left(1 + \left\|\xi \right\|_{L^2(\Omega;\,D((\lambda-A)^{\beta/2})}\right).
 \end{align}
 The It\^o isometry and Lemma \ref{fractionalprop} (ii) and (iv) -- (vii) yield
 \begin{align}\label{inequal5}
  \I_4^2 &\leq 2 \, \mathbb E \int\limits_0^{\tau_1} \left\| \left[ e^{A (\tau_2-\tau_1)} - I \right] e^{A (\tau_1-s)} G \right\|_{\mathcal L_{HS}(Q^{1/2}(H);\,H)}^2 \ds \nonumber \\
  &\quad + 2 \, \mathbb E \int\limits_{\tau_1}^{\tau_2} \left\| (\lambda-A)^{1/2-\gamma} e^{A (\tau_2-s)} (\lambda-A)^{\gamma-1/2} G \right\|_{\mathcal L_{HS}(Q^{1/2}(H);\,H)}^2 \ds \nonumber \\
  &\leq c \left[ (\tau_2-\tau_1)^{2\gamma} \int\limits_0^{\tau_1} (\tau_1-s)^{-2\gamma} \ds \right. \nonumber \\
  &\qquad \quad \left. + \int\limits_{\tau_1}^{\tau_2} (\tau_2-s)^{2\gamma-1} \ds \right] \mathbb E  \left\| (\lambda-A)^{-(\beta-1)/2} (\lambda-A)^{(\beta-1)/2} G \right\|_{\mathcal L_{HS}(Q^{1/2}(H);\,H)}^2\nonumber \\
  &\leq c \, (\tau_2-\tau_1)^{2\gamma} \, \mathbb E  \left\| G \right\|_{\mathcal L_{HS}(Q^{1/2}(H);\,D((\lambda-A)^{(\beta-1)/2}))}^2.
 \end{align}
 Using the It\^o isometry and Lemma \ref{fractionalprop} (iv) -- (vii), we get for all $\alpha \in (\gamma,3/4)$
 \begin{align}\label{inequal6}
  \I_5^2 &\leq \int\limits_0^{\tau_1} \left\| \left[ e^{A (\tau_2-\tau_1)} - I \right] (\lambda-A)^{1-\alpha} e^{A (\tau_1-s)} (\lambda-A)^\alpha N \right\|_{\mathcal L_{HS}(Q_b^{1/2}(H_b);\,H)}^2 \ds \nonumber \\
  &\quad + \int\limits_{\tau_1}^{\tau_2} \left\| (\lambda-A)^{1/2-\gamma} e^{A (\tau_2-s)} (\lambda-A)^{1/2+\gamma} N \right\|_{\mathcal L_{HS}(Q_b^{1/2}(H_b);\,H)}^2 \ds \nonumber \\
  &\leq c \left[ (\tau_2-\tau_1)^{2\gamma} \int\limits_0^{\tau_1} (\tau_1-s)^{2\alpha-2\gamma-2} \ds + \int\limits_{\tau_1}^{\tau_2} (\tau_2-s)^{2\gamma-1} \ds \right] \nonumber \\
  &\leq c \, (\tau_2-\tau_1)^{2\gamma}.
 \end{align}
 Substituting the inequalities \eqref{inequal2}--\eqref{inequal6} in \eqref{inequal1} yields the result.
\end{proof}

\subsection{Continuity of mild solutions to the Riccati equation}

We also need similar continuity properties of the mild solution $\mathcal P$ to the Riccati equation~\eqref{riccati} and the transformed version $\mathcal H = (\lambda-A)^{1-\alpha}\mathcal P$.
In the following, we use $c > 0$ as a generic constant that may changes from time to time. 

\begin{lemma}\label{riccaticontinuity1}
 Let $(\mathcal P(t))_{t \in [0,T]}$ be the mild solution of system \eqref{riccati}.
 Then there exists a constant $c > 0$ such that for all $\tau_1,\tau_2 \in [0,T]$ with $\tau_1 < \tau_2$ and all $\gamma \in (0,1)$
 \begin{equation*}
  \|\mathcal P(\tau_2)-\mathcal P(\tau_1)\|_{\mathcal L (H)} \leq c \, (\tau_2-\tau_1)^\gamma.
 \end{equation*}
\end{lemma}

\begin{proof}
 Let $y \in H$.
 We set for all $t \in [0,T]$
 \begin{align*}
  &\mathcal J(t) = \mathcal{P}(t)B R^{-1} B^* \mathcal{P}(t), &\mathcal K(t) = \mathcal{H}^*(t) \mathcal{G} R_b^{-1}\mathcal{G}^* \mathcal{H}(t) - C^*C.
 \end{align*}
 Note that the operators $\mathcal J(t)$ and $\mathcal K(t)$ are linear and bounded.
 By definition, we get
 \begin{equation}\label{in1}
  \|\mathcal P(\tau_2)y-\mathcal P(\tau_1)y\|_H \leq \I_1 + \I_2,
 \end{equation}
 where 
 \begin{align*}
  \I_1 &= \left\| \int\limits_{\tau_2}^T \left[ e^{A(s-\tau_2)} - e^{A(s-\tau_1)} \right] \mathcal J(s) e^{A(s-\tau_2)}y \, \ds \right\|_H + \left\| \int\limits_{\tau_2}^T e^{A(s-\tau_1)} \mathcal J(s) \left[ e^{A(s-\tau_2)} - e^{A(s-\tau_1)} \right] y \, \ds \right\|_H \\
  &\quad + \left\| \int\limits_{\tau_1}^{\tau_2} e^{A(s-\tau_1)}\mathcal J(s)e^{A(s-\tau_1)}y \, \ds \right\|_H 
 \end{align*}
and
 \begin{align*}
  \I_2 &= \left\| \int\limits_{\tau_2}^T \left[ e^{A(s-\tau_2)} - e^{A(s-\tau_1)} \right] \mathcal K(s) e^{A(s-\tau_2)}y \, \ds \right\|_H + \left\| \int\limits_{\tau_2}^T e^{A(s-\tau_1)}\mathcal K(s) \left[ e^{A(s-\tau_2)} - e^{A(s-\tau_1)} \right] y \, \ds \right\|_H \\
  &\quad + \left\| \int\limits_{\tau_1}^{\tau_2} e^{A(s-\tau_1)} \mathcal K(s) e^{A(s-\tau_1)}y \, \ds \right\|_H.
 \end{align*}
 Recall that the semigroup $(e^{A t})_{t \geq 0}$ is a contraction and that the operators $\mathcal P(t)$, $B$ and $R^{-1}$ are linear and bounded for all $t \in [0,T]$.
 Lemma \ref{fractionalprop} (i) and (iv) -- (vii) give us
 \begin{align}\label{in2}
  \I_1 &= \int\limits_{\tau_2}^T \left\| \left[ I - e^{A(\tau_2-\tau_1)} \right] e^{A(s-\tau_2)} \mathcal J(s) e^{A(s-\tau_2)}y \right\|_H \ds + \int\limits_{\tau_2}^T \left\| e^{A(s-\tau_1)} \mathcal J(s) \left[ I - e^{A(\tau_2-\tau_1)} \right] e^{A(s-\tau_2)} y \right\|_H \ds \nonumber \\
  &\quad + \int\limits_{\tau_1}^{\tau_2} \left\| (\lambda-A)^{1-\gamma} e^{A(s-\tau_1)} (\lambda-A)^{\gamma-1} \mathcal J(s) e^{A(s-\tau_1)}y \right\|_H \ds \nonumber \\
  &\leq c \left[ (\tau_2-\tau_1)^\gamma \int\limits_{\tau_2}^T (s-\tau_2)^{-\gamma} \, \ds + \int\limits_{\tau_1}^{\tau_2} (s-\tau_1)^{\gamma-1} \, \ds \right] \| y \|_H \nonumber \\
  &\leq c \, (\tau_2-\tau_1)^\gamma \| y \|_H.
 \end{align}
 Recall that the operators $\mathcal H(t)$, $\mathcal G$, $R_b^{-1}$ and $C$ are linear and bounded for all $t \in [0,T]$.
 Similarly to the above, we obtain
 \begin{equation}\label{in3}
  \I_2 \leq c \, (\tau_2-\tau_1)^\gamma \| y \|_H.
 \end{equation}
 Substituting the inequalities \eqref{in2} and \eqref{in3} in \eqref{in1} yields the result. 
\end{proof}

\begin{lemma}\label{riccaticontinuity2}
 Let $(\mathcal H(t))_{t \in [0,T]}$ be given by
 \begin{equation*}
  \mathcal H(t) = (\lambda-A)^{1-\alpha}\mathcal P(t)
 \end{equation*}  
 for $\alpha \in (1/2,3/4)$, where $(\mathcal P(t))_{t \in [0,T]}$ is the mild solution of system \eqref{riccati}.
 Then there exists a constant $c > 0$ such that for all $\tau_1,\tau_2 \in [0,T)$ with $\tau_1 < \tau_2$ and all $\gamma \in (0,\alpha)$,
 \begin{equation*}
  \|\mathcal H(\tau_2)-\mathcal H(\tau_1)\|_{\mathcal L (H)} \leq c \, (\tau_2-\tau_1)^\gamma.
 \end{equation*}
\end{lemma}

\begin{proof}
 Let $y \in H$.
 We set for all $t \in [0,T]$
 \begin{align*}
  &\mathcal J(t) = \mathcal{P}(t)B R^{-1} B^* \mathcal{P}(t), &\mathcal K(t) = \mathcal{H}^*(t) \mathcal{G} R_b^{-1}\mathcal{G}^* \mathcal{H}(t) - C^*C.
 \end{align*}
 Note that the operators $\mathcal J(t)$ and $\mathcal K(t)$ are linear and bounded.
 By definition, we get
 \begin{equation}\label{i1}
  \|\mathcal H(\tau_2)y-\mathcal H(\tau_1)y\|_H = \| (\lambda-A)^{1-\alpha} \mathcal P(\tau_2)y- (\lambda-A)^{1-\alpha}\mathcal P(\tau_1)y\|_H \leq \I_1 + \I_2,
 \end{equation}
 where 
 \begin{align*}
  \I_1 &= \left\| \int\limits_{\tau_2}^T (\lambda-A)^{1-\alpha} \left[ e^{A(s-\tau_2)} - e^{A(s-\tau_1)} \right] \mathcal J(s) e^{A(s-\tau_2)}y \, \ds \right\|_H \\
  &\quad + \left\| \int\limits_{\tau_2}^T (\lambda-A)^{1-\alpha} e^{A(s-\tau_1)} \mathcal J(s) \left[ e^{A(s-\tau_2)} - e^{A(s-\tau_1)} \right] y \, \ds \right\|_H \\
  &\quad + \left\| \int\limits_{\tau_1}^{\tau_2} (\lambda-A)^{1-\alpha} e^{A(s-\tau_1)} \mathcal J(s) e^{A(s-\tau_1)}y \, \ds \right\|_H
 \end{align*}
 and
 \begin{align*}
  \I_2 &= \left\| \int\limits_{\tau_2}^T (\lambda-A)^{1-\alpha} \left[ e^{A(s-\tau_2)} - e^{A(s-\tau_1)} \right] \mathcal K(s) e^{A(s-\tau_2)}y \, \ds \right\|_H \\
  &\quad + \left\| \int\limits_{\tau_2}^T (\lambda-A)^{1-\alpha} e^{A(s-\tau_1)} \mathcal K(s) \left[ e^{A(s-\tau_2)} - e^{A(s-\tau_1)} \right] y \, \ds \right\|_H \\
  &\quad + \left\| \int\limits_{\tau_1}^{\tau_2} (\lambda-A)^{1-\alpha} e^{A(s-\tau_1)} \mathcal K(s) e^{A(s-\tau_1)}y \, \ds \right\|_H.
 \end{align*}
 Recall that the semigroup $(e^{A t})_{t \geq 0}$ is a contraction and that the operators $\mathcal P(t)$, $B$ and $R^{-1}$ are linear and bounded for all $t \in [0,T]$.
 Lemma \ref{fractionalprop} (i) and (iv) -- (vii) show that
 \begin{align}\label{i2}
  \I_1 &= \int\limits_{\tau_2}^T \left\| \left[ I - e^{A(\tau_2-\tau_1)} \right] (\lambda-A)^{1-\alpha} e^{A(s-\tau_2)} \mathcal J(s) e^{A(s-\tau_2)}y \right\|_H \ds  \nonumber \\
  &\quad + \int\limits_{\tau_2}^T \left\| (\lambda-A)^{1-\alpha} e^{A(s-\tau_1)} \mathcal J(s) \left[ I - e^{A(\tau_2-\tau_1)} \right] e^{A(s-\tau_2)} y \right\|_H \ds \nonumber \\
  &\quad + \int\limits_{\tau_1}^{\tau_2} \left\| (\lambda-A)^{1-\alpha} e^{A(s-\tau_1)} \mathcal J(s) (\lambda-A)^{\alpha-\gamma} e^{A(s-\tau_1)} (\lambda-A)^{\gamma-\alpha} y \right\|_H \ds \nonumber \\
  &\leq c \left[ (\tau_2-\tau_1)^\gamma \int\limits_{\tau_2}^T (s-\tau_2)^{\alpha-1-\gamma} \, \ds \right. \nonumber \\
  &\qquad \quad \left. + (\tau_2-\tau_1)^{-\alpha+1+\gamma} \int\limits_{\tau_2}^T (s-\tau_1)^{\alpha-1} (s-\tau_2)^{\alpha-1-\gamma} \, \ds + \int\limits_{\tau_1}^{\tau_2} (s-\tau_1)^{\gamma-1} \, \ds \right] \| y \|_H \nonumber \\
  &\leq c \, (\tau_2-\tau_1)^\gamma \| y \|_H.
 \end{align}
 Since the operators $\mathcal H(t)$, $\mathcal G$, $R_b^{-1}$ and $C$ are linear and bounded for all $t \in [0,T]$, a very similar argument leads to
 \begin{equation}\label{i3}
  \I_2 \leq c \, (\tau_2-\tau_1)^\gamma \| y \|_H.
 \end{equation}
 Substituting the inequalities \eqref{i2} and \eqref{i3} in \eqref{i1} yields the result.
\end{proof}

\subsection{Discretized solution operators}
Finally, we collect some results that compare the spatially discretized solution operator $e^{A_h t}$ to the exact solution operator $e^{A t}$, and to the fully discretized time stepping operator $S_{h,\Delta t}$.

\begin{lemma}\label{spatialprop}
 There exists a constant $c>0$ such that
 \begin{enumerate}
  \item for every $y \in D((\lambda-A)^{\rho/2})$ with $\rho,r \in [0,2]$ satisfying $\rho \leq r$ and all $t>0$:
   \begin{equation*}
    \left\|e^{A t} y- e^{A_h t} P_h y \right\|_H \leq c \,  h^r t^{-(r-\rho)/2}\|y\|_{\rho/2};
   \end{equation*}
  \item for every $y \in D((\lambda-A)^{-\rho/2})$ with $\rho \in [0,1]$ and all $t>0$:
  \begin{equation*}
    \left\|e^{A t} y- e^{A_h t} P_h y \right\|_H \leq c \,  h^{2-\rho} t^{-1}\|y\|_{-\rho/2};
   \end{equation*}
  \item for every $y \in D((\lambda-A)^\alpha)$ with $\alpha \in [1/2,1]$ and all $t>0$:
  \begin{equation*}
    \left\|(\lambda-A) e^{A t} y- e^{A_h t} (\lambda-A_h) R_h y \right\|_H \leq c \,  h^{2\alpha} t^{-1} \|y\|_\alpha.
   \end{equation*}
 \end{enumerate}
\end{lemma}

\begin{proof}
 A proof of (i) can be found in \cite[Lemma 3.1]{Lord13} for $r \in \{1,2\}$.
 For $r=0$, the inequality is an immediate consequence of the fact that the semigroups $(e^{A t})_{t \geq 0}$ and $(e^{A_h t})_{t \geq 0}$ are contractions.
 The result holds for all $r \in [0,2]$ applying interpolation techniques, which is demonstrated in \cite[Theorem 3.5]{Thomee06}.
 For the assertion (ii), we can follow \cite[Lemma 3.2 (iii)]{Tambue19}.
 It remains to show (iii).
 Let us first assume that $y \in D(A)$.
 By equation \eqref{projectionequation}, Lemma \ref{fractionalprop} (iv) and (v), and claim (ii) with $\rho = 2-2\alpha$, we obtain
 \begin{align*}
  \left\|(\lambda-A) e^{A t} y- e^{A_h t} (\lambda-A_h) R_h y \right\|_H &= \left\| e^{A t} (\lambda-A) y- e^{A_h t} P_h (\lambda-A) y \right\|_H \\
  &\leq c \,  h^{2\alpha} t^{-1} \left\|  (\lambda-A)^{\alpha-1} (\lambda-A) y \right\|_H \\
  &= c \,  h^{2\alpha} t^{-1} \left\| y \right\|_\alpha.
 \end{align*}
 The above inequality holds also for every $y \in D((\lambda-A)^\alpha)$ with $\alpha \in [1/2,1]$ by standard density arguments.
 Indeed, for every $y \in D((\lambda-A)^\alpha)$, there exists a sequence $(y_k)_{k \in \mathbb N} \subset D(A)$ such that $y_k \rightarrow y$ in $D((\lambda-A)^\alpha)$ as $k \rightarrow \infty$ resulting from Lemma \ref{fractionalprop} (ii) and (iii).
 Due to Lemma \ref{fractionalprop} (iv) -- (vi), we get for each $k \in \mathbb N$
 \begin{align*}
  &\left\|(\lambda-A) e^{A t} y- e^{A_h t} (\lambda-A_h) R_h y \right\|_H \\
  &= \left\|(\lambda-A) e^{A t} (y- y_k) - e^{A_h t} (\lambda-A_h) R_h (y-y_k) + (\lambda-A) e^{A t} y_k - e^{A_h t} (\lambda-A_h) R_h y_k \right\|_H \\
  &\leq \left\|(\lambda-A)^{1-\alpha} e^{A t} (\lambda-A)^\alpha (y- y_k) \right\|_H + \left\| e^{A_h t} (\lambda-A_h) R_h (y-y_k) \right\|_H \\
  &\quad + \left\| (\lambda-A) e^{A t} y_k - e^{A_h t} (\lambda-A_h) R_h y_k \right\|_H \\
  &\leq \left( M_{1-\alpha} \, t^{\alpha-1} + c \, M_0 \right) \|y- y_k\|_\alpha + c \,  h^{2\alpha} t^{-1} \left\| y_k \right\|_\alpha \\
  &\leq \left( M_{1-\alpha} \, t^{\alpha-1} + c \, M_0 \right) \|y- y_k\|_\alpha + c \,  h^{2\alpha} t^{-1} \left\| y_k - y \right\|_\alpha + c \,  h^{2\alpha} t^{-1} \left\| y \right\|_\alpha.
 \end{align*}
 Hence, the result follows as $k \rightarrow \infty$.
\end{proof}

\begin{lemma}[Theorem 6.1, \cite{Fujita76}]\label{contraction}
 For each $m=1,...,M$, we have
 \begin{equation*}
  \left\| S_{h,\Delta t}^m \right\|_{\mathcal L(H)} \leq 1,
 \end{equation*}
 where $S_{h,\Delta t}^m$ denotes the composition of $S_{h,\Delta t}$ with itself $m$ times.
\end{lemma}

\begin{lemma}\label{timeprop}
 There exists a constant $c>0$ such that
 \begin{enumerate}
  \item for every $y \in D((\lambda-A)^{\rho/2})$ with $\rho \in [0,2]$ and each $m=0,1,...,M$:
  \begin{equation*}
   \left\| e^{A_h t_m} P_h y - S_{h,\Delta t}^m P_h y \right\|_H \leq c \,  \Delta t^{\rho/2} \|y\|_{\rho/2};
  \end{equation*}
  \item for every $y \in D((\lambda-A)^{-\rho/2})$ with $\rho \in [0,1]$ and each $m=1,...,M$:
  \begin{equation*}
   \left\| e^{A_h t_m} P_h y - S_{h,\Delta t}^m P_h y \right\|_H \leq c \,  t_m^{-1} \Delta t^{(2-\rho)/2}\|y\|_{-\rho/2};
  \end{equation*}
  \item for every $y \in D((\lambda-A)^{\alpha})$ with $\alpha \in [1/2,1]$ and each $m=1,...,M$:
  \begin{equation*}
   \left\| e^{A_h t_m} (\lambda-A_h) R_h y - S_{h,\Delta t}^m (\lambda-A_h) R_h y \right\|_H \leq c \,  t_m^{-1} \Delta t^\alpha \|y\|_\alpha.
  \end{equation*}
 \end{enumerate}
\end{lemma}

\begin{proof}
 The claims (i) and (ii) are proven in \cite[Lemma 3.3]{Tambue19}.
 It remains to show (iii).
 Let us first assume that $y \in D(A)$.
 Using equation \eqref{projectionequation} and (ii) with $\rho = 2-2\alpha$, we get
 \begin{align*}
  \left\| e^{A_h t_m} (\lambda-A_h) R_h y - S_{h,\Delta t}^m (\lambda-A_h) R_h y \right\|_H &= \left\| e^{A_h t_m} P_h (\lambda-A) y - S_{h,\Delta t}^m P_h (\lambda-A) y \right\|_H \\
  &\leq c \,  t_m^{-1} \Delta t^\alpha \left\|(\lambda-A)^{\alpha-1} (\lambda-A) y \right\|_H \\
  &= c \,  t_m^{-1} \Delta t^\alpha \left\| y \right\|_\alpha.
 \end{align*}
 The above inequality holds also for every $y \in D((\lambda-A)^\alpha)$ with $\alpha \in [1/2,1]$ by standard density arguments as demonstrated in Lemma \ref{spatialprop} (iii).
\end{proof}

\begin{lemma}\label{intspatialprop}
 There exists a constant $c>0$ such that
 \begin{itemize}
  \item[(i)] for every $y \in D((\lambda-A)^{-\rho/2})$ with $\rho \in [0,1]$ and all $t>0$:
   \begin{equation*}
    \left\| \int\limits_0^t e^{A s} y- e^{A_h s} P_h y \, \ds \right\|_H \leq c \,  h^{2-\rho} \|y\|_{-\rho/2};
   \end{equation*}
   \item[(ii)] for every $y \in D((\lambda-A)^{(\mu-1)/2})$ with $\mu \in [0,2]$ and all $t>0$:
   \begin{equation*}
    \left( \int\limits_0^t \left\| e^{A s} y- e^{A_h s} P_h y \right\|_H^2 \ds \right)^{1/2} \leq c \,  h^\mu \|y\|_{(\mu-1)/2};
   \end{equation*}
   \item[(iii)] for every $y \in D((\lambda-A)^\alpha)$ with $\alpha \in [1/2,1]$ and all $t>0$:
   \begin{equation*}
    \left\| \int\limits_0^t (\lambda-A) e^{A s} y- e^{A_h s} (\lambda-A_h) R_h y \, \ds \right\|_H \leq c \,  h^{2\alpha} \|y\|_\alpha;
   \end{equation*}
   \item[(iv)] for every $y \in D((\lambda-A)^\alpha)$ with $\alpha \in [1/2,3/2]$ and all $t>0$:
   \begin{equation*}
    \left( \int\limits_0^t \left\| (\lambda-A) e^{A s} y- e^{A_h s} (\lambda-A_h) R_h y \right\|_H^2 \ds \right)^{1/2} \leq c \,  h^{2\alpha-1} \|y\|_\alpha.
   \end{equation*}
 \end{itemize}
\end{lemma}

\begin{proof}
 The claims (i) and (ii) are shown in \cite[Lemma 3.2]{Tambue19}.
 It remains to show (iii) and (iv).
 First, we assume that $y \in D(A)$.
 Using Lemma \ref{fractionalprop} (iv) and (v), equation \eqref{projectionequation}, and (i) with $\rho = 2 - 2\alpha$, we get
 \begin{align}\label{Ineq1}
  \left\| \int\limits_0^t (\lambda-A) e^{A s} y- e^{A_h s} (\lambda-A_h) R_h y \, \ds \right\|_H &= \left\| \int\limits_0^t e^{A s} (\lambda-A) y- e^{A_h s} P_h (\lambda-A) y \, \ds \right\|_H \nonumber \\
  &\leq c \,  h^{2\alpha} \|(\lambda-A)^{\alpha-1}(\lambda-A) y\|_H \nonumber \\
  &= c \,  h^{2\alpha} \|y\|_\alpha.
 \end{align}
 Using Lemma \ref{fractionalprop} (iv) and (v), equation \eqref{projectionequation}, and (ii) with $\mu = 2\alpha-1$, we have
 \begin{align}\label{Ineq2}
  \left( \int\limits_0^t \left\| (\lambda-A) e^{A s} y- e^{A_h s} (\lambda-A_h) R_h y \right\|_H^2 \ds \right)^{1/2} &= \left( \int\limits_0^t \left\| e^{A s} (\lambda-A) y - e^{A_h s} P_h (\lambda-A) y \right\|_H^2 \ds \right)^{1/2} \nonumber \\
  &\leq c \,  h^{2\alpha-1} \|(\lambda-A)^{\alpha-1}(\lambda-A) y\|_H \nonumber\\
  &= c \,  h^{2\alpha-1} \|y\|_\alpha.
 \end{align}
 Inequality \eqref{Ineq1} holds for every $y \in D((\lambda-A)^\alpha)$ with $\alpha \in [1/2,1]$ and inequality \eqref{Ineq2} holds for every $y \in D((\lambda-A)^\alpha$ with $\alpha \in [1/2,3/2]$ by standard density arguments as shown in Lemma \ref{spatialprop} (iii).
\end{proof}

\begin{lemma}\label{inttimeprop}
 There exists a constant $c>0$ such that
 \begin{itemize}
  \item[(i)] for arbitrary small $\varepsilon>0$, every $y \in D((\lambda-A)^{-\rho/2})$ with $\rho \in [0,1]$, and each $m=1,...,M$:
   \begin{equation*}
    \left\| \sum_{k=0}^{m-1} \, \int\limits_{t_k}^{t_{k+1}} e^{A_h s} P_h y - S_{h,\Delta t}^{k+1} P_h y \, \ds \right\|_H \leq c \,  \Delta t^{(2-\rho)/2 -\varepsilon} \|y\|_{-\rho/2};
   \end{equation*}
  \item[(ii)] for arbitrary small $\varepsilon>0$, every $y \in D((\lambda-A)^{(\mu-1)/2})$ with $\mu \in [0,2]$, and each $m=1,...,M$:
   \begin{equation*}
    \left( \sum_{k=0}^{m-1} \, \int\limits_{t_k}^{t_{k+1}} \left\| e^{A_h s} P_h y - S_{h,\Delta t}^{k+1} P_h y \right\|_H^2 \ds \right)^{1/2} \leq c \,  \Delta t^{\mu/2 - \varepsilon} \|y\|_{(\mu-1)/2};
   \end{equation*}
  \item[(iii)] for arbitrary small $\varepsilon>0$, every $y \in D((\lambda-A)^\alpha)$ with $\alpha \in [1/2,1]$, and each $m=1,...,M$:
   \begin{equation*}
    \left\| \sum_{k=0}^{m-1} \, \int\limits_{t_k}^{t_{k+1}} e^{A_h s} (\lambda-A_h) R_h y - S_{h,\Delta t}^{k+1} (\lambda-A_h) R_h y \, \ds \right\|_H \leq c \,  \Delta t^{\alpha -\varepsilon} \|y\|_\alpha;
   \end{equation*}
 \end{itemize}
 \item[(iv)] for arbitrary small $\varepsilon>0$, every $y \in D((\lambda-A)^\alpha)$ with $\alpha \in [1/2,3/2]$, and each $m=1,...,M$:
   \begin{equation*}
    \left( \sum_{k=0}^{m-1} \, \int\limits_{t_k}^{t_{k+1}} \left\| e^{A_h s} (\lambda-A_h) R_h y - S_{h,\Delta t}^{k+1} (\lambda-A_h) R_h y \right\|_H^2 \ds \right)^{1/2} \leq c \,  \Delta t^{(2\alpha-1)/2 - \varepsilon} \|y\|_\alpha.
   \end{equation*}
\end{lemma}

\begin{proof}
 Assertions (i) and (ii) are proven in \cite[Lemma 3.5]{Tambue19}.
 Claims (iii) and (iv) can be obtained similarly to Lemma \ref{intspatialprop} (iii) and (iv).
\end{proof}

\section{Proof of Theorem \ref{mainresult}}\label{sec:proof}

After all the preparation in the previous section, we can now prove the main result.

\begin{proof}[Proof of Theorem \ref{mainresult}]
 The mild solution of system \eqref{stochheatcontrolled} can be rewritten $\mathbb P$-a.s.
 \begin{align*}
  y(t_m) &= e^{A t_m} \xi - \sum_{k=0}^{m-1} \int\limits_{t_k}^{t_{k+1}} e^{A (t_m-s)} B R^{-1} B^* \mathcal P(s) y(s) \, \ds - \sum_{k=0}^{m-1} \int\limits_{t_k}^{t_{k+1}} (\lambda-A) e^{A (t_m-s)} N R_b^{-1} \mathcal G^* \mathcal H(s) y(s) \, \ds \\
  &\quad + \sum_{k=0}^{m-1} \int\limits_{t_k}^{t_{k+1}} e^{A (t_m-s)} G \, \mathrm{d}W(s) + \sum_{k=0}^{m-1} \int\limits_{t_k}^{t_{k+1}} (\lambda-A) e^{A (t_m-s)} N \, \mathrm{d}W_b(s).
 \end{align*}
 Similarly, the fully discrete scheme \eqref{stochheatfullydiscr} can be rewritten $\mathbb P$-a.s
 \begin{align*}
  y_h^m &= S_{h,\Delta t}^m P_h \xi - \sum_{k=0}^{m-1} \int\limits_{t_k}^{t_{k+1}} S_{h,\Delta t}^{m-k} B_h R^{-1} B_h^* \mathcal P_h^k y_h^k \, \ds - \sum_{k=0}^{m-1} \int\limits_{t_k}^{t_{k+1}} S_{h,\Delta t}^{m-k} B_h^b R_b^{-1} \left( B_h^b \right)^* \mathcal P_h^k y_h^k \, \ds \\
  &\quad + \sum_{k=0}^{m-1} \int\limits_{t_k}^{t_{k+1}} S_{h,\Delta t}^{m-k} P_h G(t_{m-1}) \, \mathrm{d}W(s) + \sum_{k=0}^{m-1} \int\limits_{t_k}^{t_{k+1}} S_{h,\Delta t}^{m-k} B_h^b \, \mathrm{d}W_b(s).
 \end{align*}
 Therefore, we obtain
 \begin{equation}\label{ineq1}
  \|y(t_m)-y_h^m\|_{L^2(\Omega;\,H)} \leq \I_1 + \I_2 + \I_3 + \I_4 + \I_5, 
 \end{equation}
 where
 \begin{align*}
  \I_1 &= \left\| e^{A t_m} \xi - S_{h,\Delta t}^m P_h \xi \right\|_{L^2(\Omega;\,H)}, \\
  \I_2 &= \left\| \sum_{k=0}^{m-1} \int\limits_{t_k}^{t_{k+1}} e^{A (t_m-s)} B R^{-1} B^* \mathcal P(s) y(s) \, \ds - \int\limits_{t_k}^{t_{k+1}} S_{h,\Delta t}^{m-k} B_h R^{-1} B_h^* \mathcal P_h^k y_h^k \, \ds \right\|_{L^2(\Omega;\,H)}, \\
  \I_3 &= \left\| \sum_{k=0}^{m-1} \int\limits_{t_k}^{t_{k+1}} (\lambda-A) e^{A (t_m-s)} N R_b^{-1} \mathcal G^* \mathcal H(s) y(s) \, \ds - \int\limits_{t_k}^{t_{k+1}} S_{h,\Delta t}^{m-k} B_h^b R_b^{-1} \left( B_h^b \right)^* \mathcal P_h^k y_h^k \, \ds \right\|_{L^2(\Omega;\,H)}, \\
  \I_4 &= \left\| \sum_{k=0}^{m-1} \int\limits_{t_k}^{t_{k+1}} e^{A (t_m-s)} G \, \mathrm{d}W(s) - \int\limits_{t_k}^{t_{k+1}} S_{h,\Delta t}^{m-k} P_h G \, \mathrm{d}W(s) \right\|_{L^2(\Omega;\,H)}\quad \text{and} \\
  \I_5 &= \left\| \sum_{k=0}^{m-1} \int\limits_{t_k}^{t_{k+1}} (\lambda-A) e^{A (t_m-s)} N \, \mathrm{d}W_b(s) - \int\limits_{t_k}^{t_{k+1}} S_{h,\Delta t}^{m-k} B_h^b \, \mathrm{d}W_b(s) \right\|_{L^2(\Omega;\,H)}.
 \end{align*}
 Lemma \ref{spatialprop} (i) with $r=\rho=\beta$ and Lemma \ref{timeprop} (i) with $\rho=\beta$ gives us
 \begin{align}\label{ineq2}
  \I_1 &\leq \left\|e^{A t_m} \xi - e^{A_h t_m} P_h \xi \right\|_{L^2(\Omega;\,H)} + \left\|e^{A_h t_m} P_h \xi - S_{h,\Delta t}^m P_h \xi \right\|_{L^2(\Omega;\,H)} \nonumber \\
  &\leq c \, (h^\beta + \Delta t^{\beta/2}) \|\xi\|_{L^2(\Omega;\,D((\lambda-A)^{\beta/2})}.
 \end{align}
 Recall that $B_h = P_h B$.
 We have
 \begin{equation}\label{ineq3}
  \I_2 \leq \I_{2,1} + \I_{2,2} + \I_{2,3} + \I_{2,4} + \I_{2,5} + \I_{2,6},
 \end{equation}
 where
 \begin{align*}
  \I_{2,1} &= \left\| \sum_{k=0}^{m-1} \int\limits_{t_k}^{t_{k+1}} \left[e^{A (t_m-s)} -e^{A_h (t_m-s)}P_h\right] B R^{-1} B^* \mathcal P(s) y(s) \, \ds \right\|_{L^2(\Omega;\,H)}, \\
  \I_{2,2} &= \left\| \sum_{k=0}^{m-1} \int\limits_{t_k}^{t_{k+1}} \left[e^{A_h (t_m-s)} P_h - S_{h,\Delta t}^{m-k} P_h \right] B R^{-1} B^* \mathcal P(s) y(s) \, \ds \right\|_{L^2(\Omega;\,H)}, \\
  \I_{2,3} &= \left\| \sum_{k=0}^{m-1} \int\limits_{t_k}^{t_{k+1}} S_{h,\Delta t}^{m-k} B_h R^{-1} [B^*-B_h^*] \mathcal P(s) y(s) \, \ds \right\|_{L^2(\Omega;\,H)}, \\
  \I_{2,4} &= \left\| \sum_{k=0}^{m-1} \int\limits_{t_k}^{t_{k+1}} S_{h,\Delta t}^{m-k} B_h R^{-1} B_h^* [\mathcal P(s) y(s)-\mathcal P(t_k)y(t_k)] \, \ds \right\|_{L^2(\Omega;\,H)}, \\
  \I_{2,5} &= \left\| \sum_{k=0}^{m-1} \int\limits_{t_k}^{t_{k+1}} S_{h,\Delta t}^{m-k} B_h R^{-1} B_h^* \left[\mathcal P(t_k)-\mathcal P_h^k P_h\right] y(t_k) \, \ds \right\|_{L^2(\Omega;\,H)} \quad \text{and} \\
  \I_{2,6} &= \left\| \sum_{k=0}^{m-1} \int\limits_{t_k}^{t_{k+1}} S_{h,\Delta t}^{m-k} B_h R^{-1} B_h^* \mathcal P_h^k P_h \left[y(t_k)-y_h^k\right] \ds \right\|_{L^2(\Omega;\,H)}.
 \end{align*}
 Recall that the operator $\mathcal P(t)$ is linear and bounded for all $t \in [0,T]$.
 Using Lemmas \ref{mildsolbound}--\ref{riccaticontinuity1}, there exists a constant $c>0$ such that for all $\tau_1,\tau_2 \in [0,T]$ with $\tau_1 < \tau_2$ and all $\gamma \in (0,1/4)$ with $\gamma \leq \beta/2$
 \begin{align}\label{continuity1}
  \left\| \mathcal P(\tau_2) y(\tau_2)- \mathcal P(\tau_1) y(\tau_1) \right\|_{L^2(\Omega;\,H)} &\leq \left\| \mathcal P(\tau_2) \left[y(\tau_2) - y(\tau_1) \right] \right\|_{L^2(\Omega;\,H)} \nonumber \\
  &\quad + \left\| \left[ \mathcal P(\tau_2) - \mathcal P(\tau_1) \right] y(\tau_1) \right\|_{L^2(\Omega;\,H)} \nonumber \\
  &\leq c \, (\tau_2-\tau_1)^\gamma \left(1 + \|\xi\|_{L^2(\Omega;\,D((\lambda-A)^{\beta/2})}\right).
 \end{align}
 We set for all $t \in [0,T]$ and $\mathbb P$-a.s
 \begin{equation*}
  \tilde y(t) = B R^{-1} B^* \mathcal P(t) y(t).
 \end{equation*}
 By a change of variables, we get
 \begin{align*}
  \I_{2,1} &\leq \left\| \int\limits_0^{t_m} \left[e^{A (t_m-s)} -e^{A_h (t_m-s)}P_h\right] \left( \tilde y(s)-\tilde y(t_m) \right) \ds \right\|_{L^2(\Omega;\,H)} \\
  &\quad + \left\| \int\limits_0^{t_m} \left[e^{A (t_m-s)} -e^{A_h (t_m-s)}P_h\right] \tilde y(t_m) \, \ds \right\|_{L^2(\Omega;\,H)} \\
  &= \left\| \int\limits_0^{t_m} \left[e^{A s} -e^{A_h s}P_h\right] \left(\tilde y(s)-\tilde y(t_m) \right) \ds \right\|_{L^2(\Omega;\,H)} + \left\| \int\limits_0^{t_m} \left[e^{A s} -e^{A_h s}P_h\right] \tilde y(t_m) \, \ds \right\|_{L^2(\Omega;\,H)}.
 \end{align*}
 Recall that the operators $B$ and $R^{-1}$ are linear and bounded.
 Due to Lemma \ref{mildsolbound}, Lemma \ref{spatialprop} (ii) with $\rho=0$, Lemma \ref{intspatialprop} (i) with $\rho=0$, and inequality \eqref{continuity1}, we obtain for all $\gamma \in (0,1/4)$ with $\gamma \leq \beta/2$
 \begin{align}\label{ineq4}
  \I_{2,1} &\leq c \, h^2 \int\limits_0^{t_m} s^{-1} \left\| \tilde y(s)-\tilde y(t_m) \right\|_{L^2(\Omega;\,H)} \ds + c \, h^2 \left\| \tilde y(t_m) \right\|_{L^2(\Omega;\,H)} \nonumber \\
  &\leq c \, h^2 \left[\int\limits_0^{t_m} s^{\gamma-1} \ds + 1 \right] \left(1 + \|\xi\|_{L^2(\Omega;\,D((\lambda-A)^{\beta/2})}\right) \nonumber \\
  &\leq c \, h^2 \left(1 + \|\xi\|_{L^2(\Omega;\,D((\lambda-A)^{\beta/2})}\right).
 \end{align}
 We have
 \begin{align}\label{ineq5}
  \I_{2,2} &\leq \I_{2,2}^{(1)} + \I_{2,2}^{(2)},
 \end{align}
 where
 \begin{align*}
  \I_{2,2}^{(1)} &= \left\| \sum_{k=0}^{m-1} \int\limits_{t_k}^{t_{k+1}} \left[e^{A_h (t_m-s)} P_h - S_{h,\Delta t}^{m-k} P_h \right] \left(\tilde y(s)-\tilde y(t_m) \right) \ds \right\|_{L^2(\Omega;\,H)} \quad \text{and} \\
  \I_{2,2}^{(2)} &= \left\| \sum_{k=0}^{m-1} \int\limits_{t_k}^{t_{k+1}} \left[e^{A_h (t_m-s)} P_h - S_{h,\Delta t}^{m-k} P_h \right] \tilde y(t_m) \, \ds \right\|_{L^2(\Omega;\,H)}.
 \end{align*}
 By a change of variables, we obtain
 \begin{align*}
  \I_{2,2}^{(1)} &\leq \left\| \sum_{k=0}^{m-1} \int\limits_{t_k}^{t_{k+1}} \left[ I - e^{A_h (t_{k+1}-s)} \right] e^{A_h s} P_h \left(\tilde y(t_m-s)-\tilde y(t_m) \right) \ds \right\|_{L^2(\Omega;\,H)} \\
  &\quad + \left\| \sum_{k=0}^{m-1} \int\limits_{t_k}^{t_{k+1}} \left[e^{A_h t_{k+1}} P_h - S_{h,\Delta t}^{k+1} P_h \right] \left(\tilde y(t_m-s)-\tilde y(t_m) \right) \ds \right\|_{L^2(\Omega;\,H)} \quad \text{and} \\
  \I_{2,2}^{(2)} &= \left\| \sum_{k=0}^{m-1} \int\limits_{t_k}^{t_{k+1}} \left[e^{A_h s} P_h - S_{h,\Delta t}^{k+1} P_h \right] \tilde y(t_m) \, \ds \right\|_{L^2(\Omega;\,H)}.
 \end{align*}
 Note that the properties from Lemma \ref{fractionalprop} hold also for the operator $A_h$ and for the corresponding semigroup $(e^{A_h t})_{t \geq 0}$.
 Moreover, the operator $B_h$ is linear and bounded.
 Using Lemma \ref{timeprop} (ii) with $\rho = 0$ and inequality \eqref{continuity1}, we get for all $\gamma \in (0,1/4)$ with $\gamma \leq \beta/2$
 \begin{align}\label{ineq6}
  \I_{2,2}^{(1)} &\leq c \sum_{k=0}^{m-1} \int\limits_{t_k}^{t_{k+1}} (t_{k+1}-s) s^{-1} \left\| P_h \left(\tilde y(t_m-s)-\tilde y(t_m) \right) \right\|_{L^2(\Omega;\,H)} \ds \\
  &\quad + c \, \Delta t \sum_{k=0}^{m-1} \int\limits_{t_k}^{t_{k+1}} t_{k+1}^{-1} \left\| \tilde y(t_m-s)-\tilde y(t_m)  \right\|_{L^2(\Omega;\,H)} \ds \nonumber \\
  & \leq c \, \Delta t.
 \end{align}
 Due to Lemma \ref{mildsolbound} and Lemma \ref{inttimeprop} (i) with $\rho = 0$, we have
 \begin{equation}\label{ineq7}
  \I_{2,2}^{(2)} \leq c \, \Delta t^{1-\varepsilon} \left\| \tilde y(t_m) \right\|_{L^2(\Omega;\,H)} 
  \leq c \, \Delta t^{1-\varepsilon} \left(1 + \|\xi\|_{L^2(\Omega;\,D((\lambda-A)^{\beta/2})}\right).
 \end{equation}
 Substituting the inequalities \eqref{ineq6} and \eqref{ineq7} in \eqref{ineq5} yields
 \begin{equation}\label{ineq8}
  \I_{2,2} \leq c \, \Delta t^{1-\varepsilon} \left(1 + \|\xi\|_{L^2(\Omega;\,D((\lambda-A)^{\beta/2})}\right).
 \end{equation}
 Using Lemma \ref{mildsolbound}, Lemma \ref{contraction}, and inequality \eqref{inputoperatorrate} with $\rho=1$, we obtain
 \begin{equation}\label{ineq9}
  \I_{2,3} \leq c h \int\limits_0^{t_m} \left\| (\lambda-A)^{1/2} \mathcal P(s)y(s) \right\|_{L^2(\Omega;\,H)} \ds \leq c h \left(1 + \|\xi\|_{L^2(\Omega;\,D((\lambda-A)^{\beta/2})}\right).
 \end{equation}
 Lemma \ref{contraction} and inequality \eqref{continuity1} give us for all $\gamma \in (0,1/4)$ with $\gamma \leq \beta/2$
 \begin{equation}\label{ineq10}
  \I_{2,4} \leq c \sum_{k=0}^{m-1} \int\limits_{t_k}^{t_{k+1}} (s-t_k)^\gamma \ds \left(1 + \|\xi\|_{L^2(\Omega;\,D((\lambda-A)^{\beta/2})}\right) \leq c \, \Delta t^\gamma \left(1 + \|\xi\|_{L^2(\Omega;\,D((\lambda-A)^{\beta/2})}\right).
 \end{equation}
 Due to Assumption \ref{riccatiassumption} and Lemma \ref{mildsolbound}, we get
 \begin{equation}\label{ineq11}
  \I_{2,5} \leq c \, (h^2 + \Delta t) \sum_{k=0}^{m-1} \int\limits_{t_k}^{t_{k+1}} \left\| y(t_k) \right\|_{L^2(\Omega;\,H)} \ds \leq c \, (h^2 + \Delta t) \left(1 + \|\xi\|_{L^2(\Omega;\,D((\lambda-A)^{\beta/2})}\right).
 \end{equation}
 As a consequence of Lemma \ref{contraction}, we have
 \begin{equation}\label{ineq12}
  \I_{2,6} \leq c \sum_{k=0}^{m-1} \left\| y(t_k)-y_h^k \right\|_{L^2(\Omega;\,H)}.
 \end{equation}
 Substituting the inequalities \eqref{ineq4} and \eqref{ineq8}--\eqref{ineq12} in \eqref{ineq3} yields for sufficiently small $\varepsilon>0$ that
 \begin{equation}\label{ineq13}
  \I_2 \leq c \, (h + \Delta t^{\min \{1/4-\varepsilon,\beta/2\}}) \left(1 + \|\xi\|_{L^2(\Omega;\,D((\lambda-A)^{\beta/2})}\right) + c \sum_{k=0}^{m-1} \left\| y(t_k)-y_h^k \right\|_{L^2(\Omega;\,H)}.
 \end{equation}
 Recall that $B_h^b = (\lambda-A_h) R_h N$.
 Similarly as above, we get
 \begin{equation}\label{ineq14}
  \I_3 \leq \I_{3,1} + \I_{3,2} + \I_{3,3} + \I_{3,4} + \I_{3,5},
 \end{equation}
 where
 \begin{align*}
  \I_{3,1} &= \left\| \sum_{k=0}^{m-1} \int\limits_{t_k}^{t_{k+1}} \left[(\lambda-A) e^{A (t_m-s)} - e^{A_h (t_m-s)} (\lambda-A_h) R_h \right] N R_b^{-1} \mathcal G^* \mathcal H(s) y(s) \, \ds \right\|_{L^2(\Omega;\,H)}, \\
  \I_{3,2} &= \left\| \sum_{k=0}^{m-1} \int\limits_{t_k}^{t_{k+1}} \left[e^{A_h (t_m-s)} (\lambda-A_h) R_h - S_{h,\Delta t}^{m-k} (\lambda-A_h) R_h \right] N R_b^{-1} \mathcal G^* \mathcal H(s) y(s) \, \ds \right\|_{L^2(\Omega;\,H)}, \\
  \I_{3,3} &= \left\| \sum_{k=0}^{m-1} \int\limits_{t_k}^{t_{k+1}} S_{h,\Delta t}^{m-k} (\lambda-A_h) R_h N R_b^{-1} \mathcal G^* \left[\mathcal H(s) y(s)- \mathcal H(t_k) y(t_k)\right] \ds \right\|_{L^2(\Omega;\,H)},
 \end{align*}
 \begin{align*}
  \I_{3,4} &= \left\| \sum_{k=0}^{m-1} \int\limits_{t_k}^{t_{k+1}} S_{h,\Delta t}^{m-k} (\lambda-A_h) R_h N R_b^{-1} \left[ \mathcal G^* \mathcal H(t_k) - \left( B_h^b \right)^* \mathcal P_h^k P_h\right] y(t_k) \, \ds \right\|_{L^2(\Omega;\,H)}  \text{and} \\
  \I_{3,5} &= \left\| \sum_{k=0}^{m-1} \int\limits_{t_k}^{t_{k+1}} S_{h,\Delta t}^{m-k} (\lambda-A_h) R_h N R_b^{-1} \left( B_h^b \right)^* \mathcal P_h^k P_h \left[ y(t_k) - y_h^k \right] \ds \right\|_{L^2(\Omega;\,H)}.
 \end{align*}
 Recall that the operator $\mathcal H(t)$ is linear and bounded for all $t \in [0,T)$.
 Using Lemma \ref{mildsolbound}, Lemma \ref{mildsolcontinuity}, and Lemma \ref{riccaticontinuity2}, there exists a constant $c>0$ such that for all $\tau_1,\tau_2 \in [0,T)$ with $\tau_1 < \tau_2$ and all $\gamma \in (0,1/4)$ with $\gamma \leq \beta/2$
 \begin{align}\label{continuity2}
  \left\| \mathcal H(\tau_2) y(\tau_2)- \mathcal H(\tau_1) y(\tau_1) \right\|_{L^2(\Omega;\,H)} &\leq \left\| \mathcal H(\tau_2) \left[y(\tau_2) - y(\tau_1) \right] \right\|_{L^2(\Omega;\,H)} \nonumber \\
  &\quad + \left\| \left[ \mathcal H(\tau_2) - \mathcal H(\tau_1) \right] y(\tau_1) \right\|_{L^2(\Omega;\,H)} \nonumber \\
  &\leq c \, (\tau_2-\tau_1)^\gamma \left(1 + \|\xi\|_{L^2(\Omega;\,D((\lambda-A)^{\beta/2})}\right).
 \end{align}
 We set for all $t \in [0,T)$ and $\mathbb P$-a.s
 \begin{equation*}
  \overline y(t) = N R_b^{-1} \mathcal G^* \mathcal H(t) y(t).
 \end{equation*}
 By a change of variables, we obtain
 \begin{align*}
  \I_{3,1} &\leq \left\| \int\limits_0^{t_m} \left[(\lambda-A) e^{A (t_m-s)} - e^{A_h (t_m-s)} (\lambda-A_h) R_h \right] \left( \overline y(s)- \overline y(t_m)\right) \ds \right\|_{L^2(\Omega;\,H)} \\
  &\quad + \left\| \int\limits_0^{t_m} \left[(\lambda-A) e^{A (t_m-s)} - e^{A_h (t_m-s)} (\lambda-A_h) R_h \right] \overline y(t_m) \, \ds \right\|_{L^2(\Omega;\,H)} \\
  &= \left\| \int\limits_0^{t_m} \left[(\lambda-A) e^{A s} - e^{A_h s} (\lambda-A_h) R_h \right] \left( \overline y(t_m-s)- \overline y(t_m)\right) \ds \right\|_{L^2(\Omega;\,H)} \\
  &\quad + \left\| \int\limits_0^{t_m} \left[(\lambda-A) e^{A s} - e^{A_h s} (\lambda-A_h) R_h \right] \overline y(t_m) \, \ds \right\|_{L^2(\Omega;\,H)}.
 \end{align*}
 Recall that the operators $(\lambda-A)^\alpha N, R_b^{-1}, \mathcal G^*$ are linear and bounded for all $\alpha \in (0,3/4)$.
 Lemma \ref{mildsolbound}, Lemma \ref{spatialprop}~(iii) with $\alpha \in [1/2,3/4)$, inequality \eqref{continuity2}, and Lemma \ref{intspatialprop}~(iii) with $\alpha \in [1/2,3/4)$ give us for all $\gamma \in (0,1/4)$ with $\gamma \leq \beta/2$
 \begin{align}\label{ineq15}
  \I_{3,1} &\leq c \, h^{2 \alpha} \int\limits_0^{t_m} s^{-1} \left\| (\lambda-A)^\alpha \left( \overline y(t_m-s)- \overline y(t_m)\right)\right\|_{L^2(\Omega;\,H)} \ds + c \, h^{2 \alpha} \left\| (\lambda-A)^\alpha \overline y(t_m) \right\|_{L^2(\Omega;\,H)} \nonumber \\
  &\leq c \, h^{2 \alpha} \left[ \int\limits_0^{t_m} s^{\gamma-1} \ds + 1 \right] \left(1 + \|\xi\|_{L^2(\Omega;\,D((\lambda-A)^{\beta/2})}\right)\nonumber \\
  &\leq c \, h^{2 \alpha} \left(1 + \|\xi\|_{L^2(\Omega;\,D((\lambda-A)^{\beta/2})}\right).
 \end{align}
 We have
 \begin{align}\label{ineq16}
  \I_{3,2} &\leq \I_{3,2}^{(1)} + \I_{3,2}^{(2)},
 \end{align}
 where
 \begin{align*}
  \I_{3,2}^{(1)} &= \left\| \sum_{k=0}^{m-1} \int\limits_{t_k}^{t_{k+1}} \left[e^{A_h (t_m-s)} (\lambda-A_h) R_h - S_{h,\Delta t}^{m-k} (\lambda-A_h) R_h \right] \left( \overline y(s)- \overline y(t_m)\right) \ds \right\|_{L^2(\Omega;\,H)}, \\
  \I_{3,2}^{(2)} &= \left\| \sum_{k=0}^{m-1} \int\limits_{t_k}^{t_{k+1}} \left[e^{A_h (t_m-s)} (\lambda-A_h) R_h - S_{h,\Delta t}^{m-k} (\lambda-A_h) R_h \right] \overline y(t_m) \, \ds \right\|_{L^2(\Omega;\,H)}.
 \end{align*}
 By a change of variables, we get
 \begin{align*}
  \I_{3,2}^{(1)} &\leq \left\| \sum_{k=0}^{m-1} \int\limits_{t_k}^{t_{k+1}} \left[ I - e^{A_h (t_{k+1}-s)} \right] e^{A_h s} (\lambda-A_h) R_h \left( \overline y(t_m-s)- \overline y(t_m)\right) \ds \right\|_{L^2(\Omega;\,H)} \\
  &\quad + \left\| \sum_{k=0}^{m-1} \int\limits_{t_k}^{t_{k+1}} \left[e^{A_h t_{k+1}} (\lambda-A_h) R_h - S_{h,\Delta t}^{k+1} (\lambda-A_h) R_h \right] \left( \overline y(t_m-s)- \overline y(t_m)\right) \ds \right\|_{L^2(\Omega;\,H)}, \\
  \I_{3,2}^{(2)} &= \left\| \sum_{k=0}^{m-1} \int\limits_{t_k}^{t_{k+1}} \left[e^{A_h s} (\lambda-A_h) R_h - S_{h,\Delta t}^{k+1} (\lambda-A_h) R_h \right] \overline y(t_m) \, \ds \right\|_{L^2(\Omega;\,H)}.
 \end{align*}
 Recall that the operators $(\lambda-A_h), R_h$ are linear and bounded.
 Lemma \ref{timeprop} (iii) with $\alpha \in [1/2,3/4)$ and inequality \eqref{continuity2} yield for all $\gamma \in (0,1/4)$ with $\gamma \leq \beta/2$
 \begin{align}\label{ineq17}
  \I_{3,2}^{(1)} &\leq c \, \sum_{k=0}^{m-1} \int\limits_{t_k}^{t_{k+1}} (t_{k+1}-s) s^{-1} \left\| (\lambda-A_h) R_h \left( \overline y(t_m-s)- \overline y(t_m)\right) \right\|_{L^2(\Omega;\,H)} \ds \nonumber \\
  &\quad + c \, \Delta t^\alpha \sum_{k=0}^{m-1} \int\limits_{t_k}^{t_{k+1}} t_{k+1}^{-1} \left\| (\lambda-A)^\alpha N R_b^{-1} \mathcal G^* \left( \overline y(t_m-s)- \overline y(t_m)\right) \right\|_{L^2(\Omega;\,H)} \ds \nonumber \\
  &\leq c \left[ \Delta t \int\limits_0^{t_m} s^{\gamma-1} \ds + \Delta t^\alpha \int\limits_0^{t_m} s^{\gamma-1} \ds \right] \left(1 + \|\xi\|_{L^2(\Omega;\,D((\lambda-A)^{\beta/2})}\right) \nonumber \\
  &\leq c \, \Delta t^\alpha \left(1 + \|\xi\|_{L^2(\Omega;\,D((\lambda-A)^{\beta/2})}\right).
 \end{align}
 Due to Lemma \ref{mildsolbound} and Lemma \ref{inttimeprop} (iii) with $\alpha \in [1/2,3,4)$, we have
 \begin{equation}\label{ineq18}
  \I_{3,2}^{(2)} \leq c \, \Delta t^{\alpha-\varepsilon} \left\| (\lambda-A)^\alpha N R_b^{-1} \mathcal G^* \mathcal H(t_m) y(t_m) \right\|_{L^2(\Omega;\,H)} \leq c \, \Delta t^{\alpha-\varepsilon} \left(1 + \|\xi\|_{L^2(\Omega;\,D((\lambda-A)^{\beta/2})}\right).
 \end{equation}
 Substituting the inequalities \eqref{ineq17} and \eqref{ineq18} in \eqref{ineq16} yields
 \begin{equation}\label{ineq19}
  \I_{3,2} \leq c \, \Delta t^\mu \left(1 + \|\xi\|_{L^2(\Omega;\,D((\lambda-A)^{\beta/2})}\right)
 \end{equation}
 with $\mu \in (0,3/4)$.
 By Lemma \ref{contraction} and inequality \eqref{continuity2}, we get for all $\gamma \in (0,1/4)$ with $\gamma \leq \beta/2$
 \begin{equation}\label{ineq20}
  \I_{3,3} \leq c \sum_{k=0}^{m-1} \int\limits_{t_k}^{t_{k+1}} (s-t_k)^\gamma \ds \left(1 + \|\xi\|_{L^2(\Omega;\,D((\lambda-A)^{\beta/2})}\right) \leq c \, \Delta t^\gamma \left(1 + \|\xi\|_{L^2(\Omega;\,D((\lambda-A)^{\beta/2})}\right).
 \end{equation}
 Using Lemma \ref{mildsolbound}, Lemma \ref{contraction}, Assumption \ref{riccatiassumption}, we have
 \begin{equation}\label{ineq21}
  \I_{3,4} \leq c \left( h + \Delta t^{1/4} \right) \left(1 + \|\xi\|_{L^2(\Omega;\,D((\lambda-A)^{\beta/2})}\right).
 \end{equation}
 Lemma \ref{contraction} gives us
 \begin{equation}\label{ineq22}
  \I_{3,5} \leq c \sum_{k=0}^{m-1} \left\| y(t_k)-y_h^k \right\|_{L^2(\Omega;\,H)}.
 \end{equation}
 Substituting the inequalities \eqref{ineq15} and \eqref{ineq19}--\eqref{ineq22} in \eqref{ineq14} yields for sufficiently small $\varepsilon>0$ that
 \begin{equation}\label{ineq23}
  \I_3 \leq c \left( h + \Delta t^{\min \{1/4-\varepsilon,\beta/2\}} \right) \left(1 + \|\xi\|_{L^2(\Omega;\,D((\lambda-A)^{\beta/2})}\right) + c \sum_{k=0}^{m-1} \left\| y(t_k)-y_h^k \right\|_{L^2(\Omega;\,H)}.
 \end{equation}
 We set $S(t) = S_{h,\Delta t}^k$ if $t \in [t_{k-1}, t_k)$ for each $k=1,...,M$.
 The It\^o isometry and a change of variables gives us
 \begin{align*}
  \I_4 &\leq \left( \mathbb E \left\| \sum_{k=0}^{m-1} \int\limits_{t_k}^{t_{k+1}} \left[ e^{A (t_m-s)} - e^{A_h (t_m-s)} P_h \right] G \, \mathrm{d}W(s) \right\|_H^2 \right)^{1/2} \\
  &\quad + \left( \mathbb E \left\| \sum_{k=0}^{m-1} \int\limits_{t_k}^{t_{k+1}} \left[ e^{A_h (t_m-s)} P_h - S_{h,\Delta t}^{m-k} P_h \right] G \, \mathrm{d}W(s) \right\|_H^2 \right)^{1/2} \\
  &= \left( \mathbb E \left\| \int\limits_0^{t_m} \left[ e^{A (t_m-s)} - e^{A_h (t_m-s)} P_h \right] G \, \mathrm{d}W(s) \right\|_H^2 \right)^{1/2} \\
  &\quad + \left( \mathbb E \left\| \int\limits_0^{t_m} \left[ e^{A_h (t_m-s)} P_h - S(t_m-s) P_h \right] G \, \mathrm{d}W(s) \right\|_H^2 \right)^{1/2} \\
  &= \left(\mathbb E  \int\limits_0^{t_m} \left\| \left[ e^{A s} - e^{A_h s} P_h \right] G \right\|_{\mathcal L_{HS}(Q^{1/2}(H);\,H)}^2 \ds \right)^{1/2} \\
  &\quad + \left(\mathbb E \sum_{k=0}^{m-1} \int\limits_{t_k}^{t_{k+1}} \left\| \left[ e^{A_h s} P_h - S_{h,\Delta t}^{k+1} P_h \right] G \right\|_{\mathcal L_{HS}(Q^{1/2}(H);\,H)}^2 \ds \right)^{1/2}.
 \end{align*}
 By Lemma \ref{intspatialprop} (ii) with $\mu = \beta$ and Lemma \ref{inttimeprop} (ii) with $\mu = \beta - 2\varepsilon$, we obtain
 \begin{align}\label{ineq24}
  \I_4 &\leq c \left(h^\beta +\Delta t^{\beta/2} \right) \left(\mathbb E \left\| (\lambda-A)^{-\varepsilon}(\lambda-A)^{(\beta-1)/2} G \right\|_{\mathcal L_{HS}(Q^{1/2}(H);\,H)}^2 \right)^{1/2} \nonumber \\
  &\leq c \left(h^\beta +\Delta t^{\beta/2} \right) \left(\mathbb E \left\| G \right\|_{\mathcal L_{HS}(Q^{1/2}(H);\,D((\lambda-A)^{(\beta-1)/2}))}^2 \right)^{1/2}.
 \end{align}
 Similarly, we have
 \begin{align*}
  \I_5 &\leq \left( \int\limits_0^{t_m} \left\| \left[ (\lambda-A) e^{A s} - e^{A_h s} (\lambda-A_h) R_h \right] N \right\|_{\mathcal L_{HS}(Q_b^{1/2}(H_b);\,H)}^2 \ds \right)^{1/2} \\
  &\quad + \left( \sum_{k=0}^{m-1} \int\limits_{t_k}^{t_{k+1}} \left\| \left[ e^{A_h s} (\lambda-A_h) R_h - S_{h,\Delta t}^k (\lambda-A_h) R_h \right] N \right\|_{\mathcal L_{HS}(Q_b^{1/2}(H_b);\,H)}^2 \ds \right)^{1/2},
 \end{align*}
 resulting from the It\^o isometry and a change of variables.
 Lemma \ref{intspatialprop} (iv) and Lemma \ref{inttimeprop} (iv), both with $\alpha \in [1/2,3,4)$, gives us for sufficiently small $\varepsilon>0$ that
 \begin{equation}\label{ineq25}
  \I_5 \leq c \left(h^{1/2-\varepsilon} +\Delta t^{1/4-\varepsilon} \right) \left\| (\lambda-A)^\alpha N \right\|_{\mathcal L_{HS}(Q_b^{1/2}(H_b);\,H)} \leq c \left(h^{1/2-\varepsilon} +\Delta t^{1/4-\varepsilon} \right).
 \end{equation}
 Substituting the inequalities \eqref{ineq2}, \eqref{ineq13}, \eqref{ineq23}, \eqref{ineq24}, and \eqref{ineq25} in \eqref{ineq1} yields for sufficiently small $\varepsilon>0$
 \begin{align*}
  \|y(t_m)-y_h^m\|_{L^2(\Omega;\,H)} &\leq c \left(h^{\min\{1/2-\varepsilon,\beta\}} +\Delta t^{\min\{1/4-\varepsilon,\beta/2\}} \right) \left(1 + \|\xi\|_{L^2(\Omega;\,D((\lambda-A)^{\beta/2})}\right) \\
  &\quad + c \sum_{k=0}^{m-1} \left\| y(t_k)-y_h^k \right\|_{L^2(\Omega;\,H)}.
 \end{align*}
By applying a discrete version of Gr\"{o}nwall's inequality, see \cite{Clark87},  we therefore get
 \begin{equation*}
  \|y(t_m)-y_h^m\|_{L^2(\Omega;\,H)} \leq c \left(h^{\min\{1/2-\varepsilon,\beta\}} +\Delta t^{\min\{1/4-\varepsilon,\beta/2\}} \right) \left(1 + \|\xi\|_{L^2(\Omega;\,D((\lambda-A)^{\beta/2})}\right),
\end{equation*}
for sufficiently small $\varepsilon>0$.
\end{proof}

\section{Numerical experiments}\label{sec:numericalexperiments}

In order to illustrate the proposed method and the bounds given in Theorem~\ref{mainresult} we have implemented the algorithm in MATLAB\footnote{Full code available at \url{www.tonystillfjord.net}.} and performed a number of numerical experiments on a two-dimensional linear quadratic control problem with noise. We ran all the experiments on one node of the Mechthild computing cluster at the Max Planck Institute Magdeburg. Such a node consists of two Intel Xeon Skylake Silver 4110 processors with 8 cores/CPU, a clockrate of 2.1 GHz and 384 GB RAM.

\subsection{Implementation}
Let $\{ \phi_k^h \}_{k=1}^{N_h}$ be the standard finite element basis of $Y_h$, consisting of the piecewise linear so-called hat-functions. These take the value 1 at the $k$-th node of $\mathcal T_h$ and are zero at all other nodes. Then for $y_h \in Y_h$ we have $y_h = \sum_{k=1}^{N_h}{ \yb_k \phi_k^h }$ for some coefficients $\{\yb_k\}_{k=1}^{N_h}$. Similarly, let the distributed noise $P_h G \, \delta W^m$ with $G = I$ and the boundary noise $B_h^b \, \delta W_b^m$ be represented by the coefficient vectors $\delta \Wb^m$ and $\delta \Wb_b^m$, respectively. Using these representations in~\eqref{stochheatfullydiscr} and testing with $\phi^h_j$ shows that~\eqref{stochheatfullydiscr} is equivalent to 
\begin{equation} \label{stochheat_FEMbasis}
  \begin{aligned}
  \sum_{k=1}^{N_h}{\yb_k^{m}  \iprod{( I - \Delta t A_h) \phi^h_k}{\phi^h_j} }
&= \sum_{k=1}^{N_h}{ \yb_k^{m-1} \Big(  \iprod{\phi^h_k}{\phi^h_j}  - \Delta t  \iprod{B_h R^{-1} B_h^* \mathcal P_h^{m-1} \phi^h_k}{\phi^h_j} }\\
  &\quad - \Delta t  \iprod{B_h^b R_b^{-1} (B_h^b)^*\mathcal P_h^{m-1} \phi^h_k}{\phi^h_j} \Big)\\
  &\quad + \sum_{k=1}^{N_h}{ \big((\delta \Wb^m)_k + (\delta \Wb_b^m)_k \big) \iprod{ \phi^h_k}{\phi^h_j}} , 
\end{aligned}
\end{equation}
for $j, k = 1,\ldots, N_h$. To simplify this, we introduce the mass matrix $\Mb$, the stiffness matrix $\Ab$, the distributed and boundary input matrices $\Bb$ and $\Bb^b$, the output matrix $\Cb$ and the weighting matrices $\Rb$, $\Rb^b$ and $\Qb$, satisfying
\begin{align*}
  \Mb_{i,j} &= \iprod{\phi^h_j}{\phi^h_i} & \Ab_{i,j} &= \iprod{ A_h \phi^h_j}{\phi^h_i} \\  
  \Bb_{i,j} &= \iprod{ B_h \phi^h_j}{\phi^U_i} & \Bb^b_{i,j} &= \iprod{ (\lambda-A_h) R_h N \phi^h_j}{\phi^V_i} &  \Cb_{i,j} &= \iprod{ C_h \phi^h_j}{\phi^Z_i} \\
  \Rb_{i,j} &= \iprod{ \phi^U_j}{\phi^U_i} & \Rb^b_{i,j} &= \iprod{ \phi^V_j}{\phi^V_i} &  \Qb_{i,j} &= \iprod{ \phi^Z_j}{\phi^Z_i} .
\end{align*}
Here, $\{\phi^U_i\}$, $\{\phi^V_i\}$ and $\{\phi^Z_i\}$ denote orthonormal bases for the input and output spaces $\bar{U}$, $\bar{V}$ and $Z$, respectively. We omit the dependency on $h$ to reduce notational clutter. 

The matrices given above were all generated using the FreeFEM++ library\footnote{Available at \url{https://freefem.org/}.}, see~\cite{FreeFem}, and then imported to MATLAB. With these at hand, we can first rewrite~\eqref{riccatisemidiscr} as the matrix-valued equation
\begin{equation}\label{riccatisemimatrix}
\begin{aligned}
   \Mb \frac{\mathrm{d}}{\dt} \Pb(t) \Mb &= -\Ab \Pb(t) \Mb - \Mb \Pb(t) \Ab +  \Mb \Pb(t)\Bb \Rb^{-1} \Bb^T \Pb(t) \Mb \\
   &\quad + \Mb \Pb(t) \Bb^b (\Rb^b)^{-1} (\Bb^b)^T \Pb(t) \Mb - \Cb^T \Qb \Cb,\\
   \Pb(T) &= 0,
\end{aligned}
\end{equation}
where $\Pb$ denotes the matrix representation of $\mathcal P_h$ satisfying
\begin{equation*}
  \mathcal P_h z = \sum_{i,j=1}^{N_h}{\Pb_{i,j}\iprod{z}{\phi^h_j} \phi^h_i},
\end{equation*}
see, e.g.,~\cite{MalPS18}.
Further denote the coefficients at time $t_m$ by $\Pb^m$. Then we can rewrite~\eqref{stochheat_FEMbasis} as an equation for the coefficients $\yb_k^m$ as
\begin{align*}
\yb^{m+1} &= \Sb_{h, \Delta t} \Mb \yb^{m} - \Delta t \Sb_{h, \Delta t} \Bb \Rb^{-1} \Bb^T \Pb^m \Mb \yb^{m} \\
    &\quad - \Delta t \Sb_{h, \Delta t} \Bb^b (\Rb^b)^{-1} (\Bb^b)^T \Pb^m \Mb \yb^{m} \\
  &\quad +  \Sb_{h, \Delta t} \Mb \big( \delta \Wb^m + \delta \Wb_b^m\big) , 
\end{align*}
where $\Sb_{h, \Delta t} = (\Mb -\Delta t \Ab)^{-1}$. Note the similarity to~\eqref{stochheatfullydiscr}, with $\Mb$ taking on the role of the identity operator. 

Since~\eqref{riccatisemimatrix} is matrix-valued, numerically approximating its solution for reasonably fine spatial discretizations is unfeasible unless we can utilize features such as low-rank structure. For this reason, we assume that the input operators are of the form $\R^{n_u} \ni u \mapsto u_1 \Psi_1 + \cdots + u_{n_u} \Psi_{n_u}$ with $\Psi_j \in H$ and $n_u \ll N_h$. Similarly, we assume that the output operator $C : H \mapsto \R^{n_z}$ with $n_z \ll N_h$. This means that $\Bb$, $\Bb^b$ and $\Cb$ are rectangular matrices, which typically leads to a solution $\Pb$ of low numerical rank. See e.g.~\cite{Sti18} for supporting theory in the operator-valued setting. In order to approximate the solution, we apply the Strang splitting scheme~\cite{Sti18a} available in the MATLAB package DREsplit\footnote{Available at \url{www.tonystillfjord.net}.}. We note that Strang splitting is a second-order method, which means that we get a more accurate approximation in time than what we need according to Assumption~\ref{riccatiassumption}. It is, however, essentially as cheap to apply as the corresponding first-order scheme, which is why we use it.

Generating the noise can be done in many ways. Since we only consider rectangular domains in our experiments, we compute samples of the distributed noise using FFT techniques as outlined in~\cite[Chapter 10]{LordPowellShardlow}. In particular, we assume that the eigenvalues $\lambda_{j,k}$ and corresponding eigenvectors $\varphi_{j,k}$ of the covariance operator $Q$ are given by 
\begin{equation*}
  \lambda_{j,k} = (j^2 + k^2)^{-\beta - \epsilon} \quad \text{and} \quad \varphi_{j,k}(x_1,x_2) = \cos(j \pi x_1) \cos(k \pi x_2) 
\end{equation*}
with $\beta = 1$ and $\epsilon = 10^{-4}$. Then the increments $\delta W^m = W(t_{m}) - W(t_{m-1})$ are given by
\begin{equation*}
  \delta W^m \approx \sqrt{\Delta t} \sum_{j,k=0}^{N}{\sqrt{\lambda_{j,k}} \varphi_{j,k} \xi^m_j},
\end{equation*}
where $\xi^n_j$ are the i.i.d. increments of N(0,1) Gaussian distribution~\cite{LordPowellShardlow}. This leads to noise satisfying Assumption~\ref{noiseoperator}. We note that the sum should actually go to infinity, and the truncation to $(N+1)^2$ terms represents a discretization. We use $N = N_h$ in our experiments, which means that the truncation does not affect the convergence order~\cite{Yan05}. 

A similar procedure could conceivably be followed for the boundary noise. However, we found it simpler to express the one-dimensional noise $\delta W_{b,m}$ on each of the edges as $\sqrt{\Delta t} \sum_{k=0}^{N}{\lambda_k \cos(k \pi x)}$ with $x \in [0,1]$ and $\lambda_k = k^{-\beta - \epsilon}$. Then the map $N$ can be explicitly constructed by using the observation that the function 
\begin{equation*}
  \rho(x_1,x_2) = -\frac{\cos(k \pi x_1) \cosh(c (1-x_2))}{c \sinh(c)} \quad \text{with} \quad c = \sqrt{\lambda + k^2 \pi^2}
\end{equation*}
satisfies $\frac{\mathrm{d}}{dx_1} \rho = 0$ at $x_1 = 0$ and $x_1 = 1$, $\frac{\mathrm{d}}{dx_2} \rho = 0$ at $x_2 = 1$ and $\frac{\mathrm{d}}{dx_2} \rho = \cos(k \pi x_1)$ at $x_2 = 0$. Further, it satisfies $\lambda \rho = \Delta \rho$ in the interior of the domain.
 The constructions for the other parts of the boundary are similar. Summing up the four parts gives then the solution of \eqref{neumann}. We then computed the Ritz projections of these functions in FreeFem++ by solving $\iprod{R_h \rho}{\phi}_{Y} = \iprod{\rho}{\phi}_{Y}$ for $\phi \in Y_h$. Finally, the resulting coefficient vectors were multiplied with $\lambda \Mb - \Ab$.

The latter construction was also used for the boundary input operator, by computing $N$ applied to the constant function $1$ on the boundary. This requires no further calculations, since it corresponds to the first eigenvector.

\subsection{Test problem}
For simplicity, we consider the problem on the unit square $\mathcal D = [0, 1]^2$. We let the distributed control operator $B: \R^{n_u} \mapsto L^2(\mathcal D)$ be defined by 
\begin{equation*}
  B u = u_1 \Psi_{p^1} + \cdots + u_{n_u} \Psi_{p^{n_u}},
\end{equation*}
where $p^j = (p^j_1, p^j_2)$ are points in the plane and $\Psi_{p^j} (x_1,x_2) = \exp{-200(x_1-p^j_1)^2 - 200(x_2-p^j_2)^2}$. The interpretation of this is that we have heat sources with high intensity at $p^j$ and tapering off exponentially as we move away radially from $p^j$. The locations of these points are illustrated in Figure~\ref{fig:testproblem}. We note that $B \in \mathcal{L}(\R^{n_u}, L^2(\mathcal D))$. For this example, we picked $n_u = 9$. For the boundary control, we consider a single boundary condition $\frac{\partial}{\partial \nu} y(t,x) = v$ with $v \in \R$.

As output we take the operator 
\begin{equation*}
  C y = 10^{2} \int_{\mathcal D}{ y(x) \chi_{T_1}(x) + \cdots + y(x) \chi_{T_{n_z}}(x) \, \mathrm{d}x },
\end{equation*}
where $\chi_S$ denotes the characteristic function of the set $S$ and $T_j$ denote different areas, illustrated in Figure~\ref{fig:testproblem}. Thus we attempt to control the mean value of the solution in these areas. We note that $C \in \mathcal{L}(H, \R^{n_z})$. Here, $n_z = 3$.

Finally, we use a diffusion coefficient of $10^{-2}$, $\lambda = 1$, and the scaling factors $R = 10^{-2}$ and $R_b = 25$. The latter was chosen such that the distributed and boundary controls influence the solution to a similar extent.

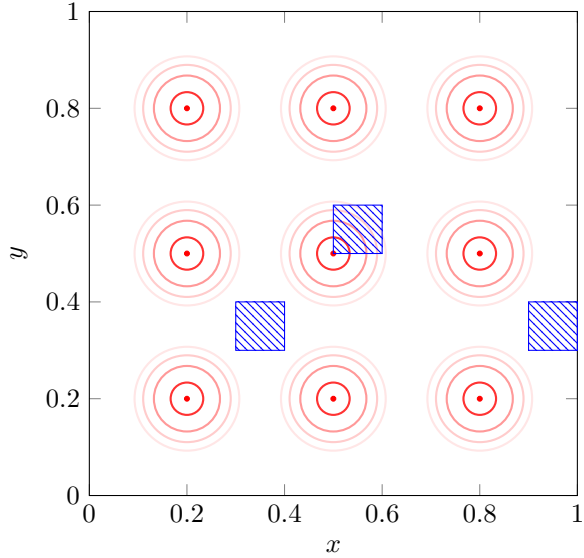
\begin{figure}
  \centering

  \begin{tikzpicture}
    \begin{axis}[xmin = 0, xmax=1, ymin = 0, ymax = 1, xlabel={$x$}, ylabel={$y$}, width = 8cm, height=8cm]
      % % Distributed inputs
      % \foreach \y in {0.1, 0.4, 0.7} {
      %   \foreach \x in {0.1, 0.4, 0.7} {
      %     % Ridiculous construction because you can't properly use foreach's in axes
      %     \edef\temp{\noexpand\filldraw [draw=red, pattern color=red,pattern=north east lines] (\x, \y) rectangle (\x+0.2, \y + 0.2);}
      %     \temp
      %   }
      %   }
    %  real[int] xs2 = [0.2, 0.2, 0.2, 0.5, 0.5, 0.5, 0.8, 0.8, 0.8];
    %  real[int] ys2 = [0.2, 0.5, 0.8, 0.2, 0.5, 0.8, 0.2, 0.5, 0.8];
      %   Distributed inputs
      \foreach \y in {0.2, 0.5, 0.8} {
        \foreach \x in {0.2, 0.5, 0.8} {
          % Ridiculous construction because you can't properly use foreach's in axes
          % \edef\temp{\noexpand\filldraw [draw=red, pattern color=red,pattern=north east lines] (\x, \y) circle [radius=0.1];}
          % 0    0.0589    0.0833    0.1020
          \edef\temp{\noexpand\filldraw [draw=red, fill=red] (\x, \y) circle [radius=0.005];}
          \temp
          \edef\temp{\noexpand\draw [draw=red!80!white, thick] (\x, \y) circle [radius=0.0334];} \temp
          \edef\temp{\noexpand\draw [draw=red!40!white, thick] (\x, \y) circle [radius=0.0677];} \temp
          \edef\temp{\noexpand\draw [draw=red!20!white, thick] (\x, \y) circle [radius=0.0897];} \temp
          \edef\temp{\noexpand\draw [draw=red!10!white, thick] (\x, \y) circle [radius=0.1073];} \temp                   
        }
      }
      % Outputs
      \filldraw [draw=blue, pattern color=blue,pattern=north west lines] (0.3, 0.3) rectangle (0.4, 0.4);
      \filldraw [draw=blue, pattern color=blue,pattern=north west lines] (0.5, 0.5) rectangle (0.6, 0.6);
      \filldraw [draw=blue, pattern color=blue,pattern=north west lines] (0.9, 0.3) rectangle (1.0, 0.4);
      
    \end{axis}

  \end{tikzpicture}

  \caption{Locations of distributed inputs (red) and outputs (blue, shaded) in the test problem. The red lines indicate intensities of 0.8, 0.4, 0.2 and 0.1, respectively.}
  \label{fig:testproblem}
\end{figure}

\subsection{Results} \label{sec:results}
We first verify Assumption~\ref{riccatiassumption} by computing the errors 
\begin{equation*}
  \|P^{ref}_h(0) - P_{h}^{0}\|_{L(H)} \quad \text{and} \quad \|(B^b_h)^*P^{ref}_h(0) - (B^b_h)^*P_{h}^{0}\|_{L(H)}
\end{equation*}
for different choices of $h$ and $\Delta t$. We first choose $h = 2^{-6}$ and $\Delta t = 2^{j + 2}$, $j = 1, \ldots, 7$, with the reference solution $P^{ref}_h$ having the same $h$ and $\Delta t = 2^{10}$. The result is shown in Figure~\ref{fig:riccatierror} (left), and shows clear second-order temporal convergence, as expected. We then choose $\Delta t = 2^{9}$ and take $h = 2^{j}$, $j = 1, \ldots, 6$, with the reference solution $P^{ref}_h$ having the same $\Delta t$ and $h = 2^{7}$. The result is shown in Figure~\ref{fig:riccatierror} (right) and also demonstrates second-order spatial convergence except for the first few coarse discretizations.

\begin{figure}
  \centering

\begin{tikzpicture}

  \begin{axis}[%
    name=ax1,
width=4.755cm,
height=5cm,
scale only axis,
xmode=log,
xmin=0.00390625,
xmax=0.25,
xminorticks=true,
xlabel style={font=\color{white!15!black}},
xlabel={$\Delta t$},
ymode=log,
ymin=1e-05,
ymax=10,
yminorticks=true,
ylabel style={font=\color{white!15!black}},
ylabel={Error},
axis background/.style={fill=white},
title style={font=\bfseries},
title={Temporal errors},
]
\addplot [color=black, mark=asterisk, mark options={solid, black, scale=1.5}]
  table[row sep=crcr]{%
0.25	0.285701903979202\\
0.125	0.0974989755596921\\
0.0625	0.0275916478168319\\
0.03125	0.0071895297763269\\
0.015625	0.00180523417388193\\
0.0078125	0.000432463283120534\\
0.00390625	8.66676688814847e-05\\
};

\addplot [color=black, mark=square, mark options={solid, black}]
  table[row sep=crcr]{%
0.25	0.956357970759475\\
0.125	0.269082287771601\\
0.0625	0.0687114122474784\\
0.03125	0.0173075373354389\\
0.015625	0.00428699598748703\\
0.0078125	0.00102162385215429\\
0.00390625	0.00020436473927703\\
};

\addplot [color=black, dashed]
  table[row sep=crcr]{%
0.25	0.167415594415743\\
0.125	0.0418538986039358\\
0.0625	0.0104634746509839\\
0.03125	0.00261586866274598\\
0.015625	0.000653967165686496\\
0.0078125	0.000163491791421624\\
0.00390625	0.000040872947855406\\
};\label{Odt2}

\end{axis}

\begin{axis}[%
  width=4.755cm,
height=5cm,
at={(ax1.south east)},
xshift=1cm,
scale only axis,
xmode=log,
xmin=0.01,
xmax=1,
xminorticks=true,
xlabel style={font=\color{white!15!black}},
xlabel={$h$},
ymode=log,
ymin=0.01,
ymax=16.3039166220533,
yminorticks=true,
axis background/.style={fill=white},
title style={font=\bfseries},
title={Spatial errors},
]
\addplot [color=black, mark=asterisk, mark options={solid, black, scale=1.5}]
  table[row sep=crcr]{%
0.5	0.931539600082039\\
0.25	0.703887550796688\\
0.125	0.498422552853944\\
0.0625	0.225493477453971\\
0.03125	0.066194626821818\\
0.015625	0.0149373778985669\\
}; \label{P}

\addplot [color=black, mark=square, mark options={solid, black}]
  table[row sep=crcr]{%
0.5	16.3039166220533\\
0.25	2.73278048643564\\
0.125	3.53489380905711\\
0.0625	0.430600712668392\\
0.03125	0.0990998191031103\\
0.015625	0.019746147801227\\
}; \label{BP}

\addplot [color=black, dashed]
  table[row sep=crcr]{%
0.5	0.47799609275414\\
0.25	0.23899804637707\\
0.125	0.119499023188535\\
0.0625	0.0597495115942675\\
0.03125	0.0298747557971338\\
0.015625	0.0149373778985669\\
}; \label{Oh}

\addplot [color=black]
  table[row sep=crcr]{%
0.5	15.2958749681325\\
0.25	3.82396874203312\\
0.125	0.955992185508281\\
0.0625	0.23899804637707\\
0.03125	0.0597495115942675\\
0.015625	0.0149373778985669\\
}; \label{Oh2}

\end{axis}
\end{tikzpicture}%
 % Legend in caption instead of in a proper legend, because any placement was bad in at least one way.
  \caption{The errors $\|P^{ref}_h(0) - P_{h}^{0}\|_{L(H)}$ (\ref{P}) and  $\|(B^b_h)^*P^{ref}_h(0) - (B^b_h)^*P_{h}^{0}\|_{L(H)}$ (\ref{BP}) for the various discretizations outlined in Section~\ref{sec:results}. Reference lines: $\ordo(\Delta t^2)$ (\ref{Odt2}), $\ordo(h)$ (\ref{Oh}), $\ordo(h^2)$ (\ref{Oh2}). }
  \label{fig:riccatierror}
\end{figure}
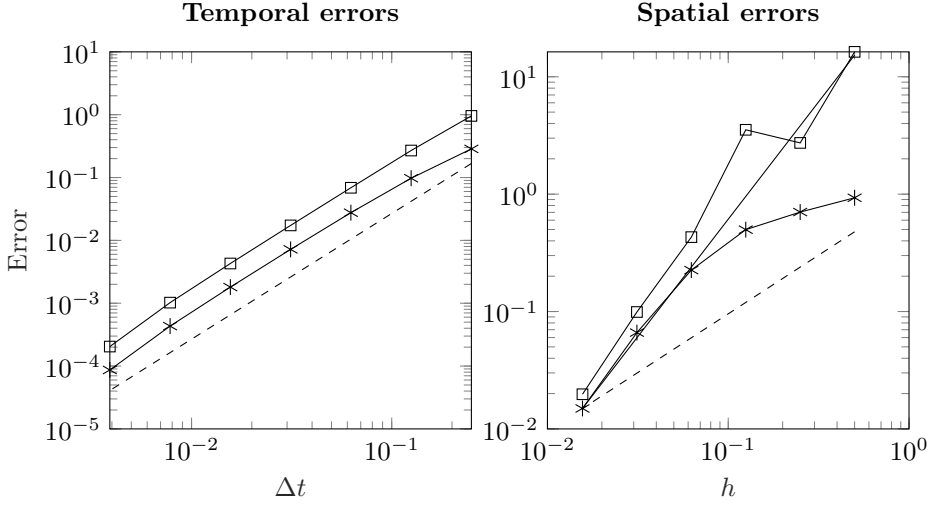

Next, we check Theorem~\ref{mainresult}. By choosing $h = \Delta t^2$, the expected error is
$\ordo\big(\Delta t^{1/4}\big)$ and there is only one parameter to adjust. We therefore choose $h = 2^{-2j}$ and $\Delta t = 2^{-j}$ for $j = 1, \ldots, 6$, and compute a reference solution with $j = 7$. We start by computing the noise for the finest discretization first. Then for each coarser discretization, we add up the temporal increments and compute the $L^2$-projection onto the coarser space. In this way we use the same noise for all the discretizations of each of the $100$ sample paths. The resulting errors measured at $t=T$ are shown in Figure~\ref{fig:main_error}, both for the controlled system and for the corresponding uncontrolled system where $b = v = 0$. We can observe that they decrease with a rate which is decidedly less than $1/2$ and close to $1/4$. Since our theoretical bound is for the worst-case situation, this is fully in line with Theorem~\ref{mainresult}.

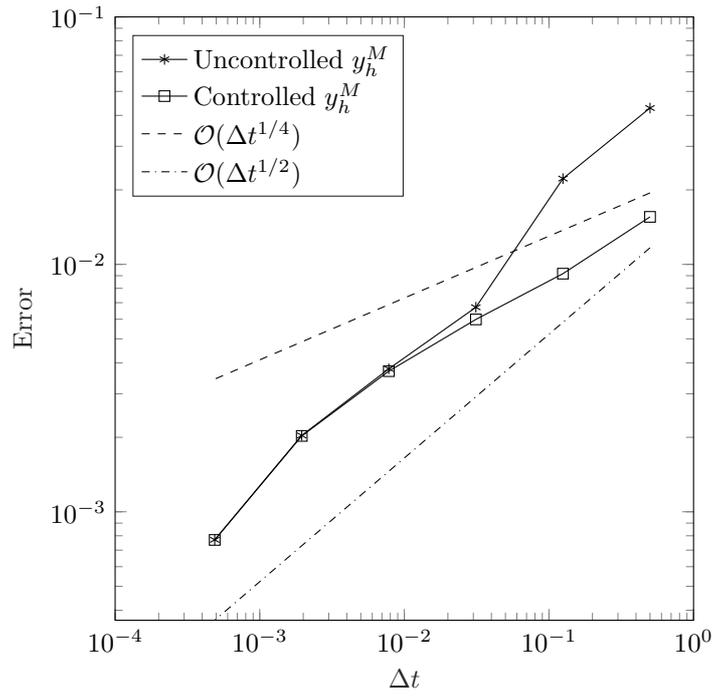
\begin{figure}
  \centering

\begin{tikzpicture}

\begin{axis}[%
  width=7.607cm,
height=8cm,
at={(0cm,0cm)},
scale only axis,
xmode=log,
xmin=0.0001,
xmax=1,
xminorticks=true,
xlabel style={font=\color{white!15!black}},
xlabel={$\Delta t$},
ymode=log,
ymin=0.000364624640554007,
ymax=0.1,
yminorticks=true,
ylabel style={font=\color{white!15!black}},
ylabel={Error},
axis background/.style={fill=white},
legend style={at={(0.03,0.97)}, anchor=north west, legend cell align=left, align=left, draw=white!15!black}
]
\addplot [color=black, mark=asterisk, mark options={solid, black}]
  table[row sep=crcr]{%
0.5	0.0428405551365907\\
0.125	0.0222144504551448\\
0.03125	0.00670691151201929\\
0.0078125	0.0037832787354787\\
0.001953125	0.00203493742969353\\
0.00048828125	0.000770725695976875\\
};
\addlegendentry{Uncontrolled $y_{h}^M$}

\addplot [color=black, mark=square, mark options={solid, black}]
  table[row sep=crcr]{%
0.5	0.015557317996971\\
0.125	0.00917082001599883\\
0.03125	0.00599058928852727\\
0.0078125	0.00370197541646373\\
0.001953125	0.00202560664812812\\
0.00048828125	0.000769757748606118\\
};
\addlegendentry{Controlled $y_{h}^M$}

\addplot [color=black, dashed]
  table[row sep=crcr]{%
0.5	0.0194466474962137\\
0.125	0.0137508563159171\\
0.03125	0.00972332374810686\\
0.0078125	0.00687542815795856\\
0.001953125	0.00486166187405343\\
0.00048828125	0.00343771407897928\\
};
\addlegendentry{$\ordo(\Delta t^{1/4})$}

\addplot [color=black, dashdotted]
  table[row sep=crcr]{%
0.5	0.0116679884977282\\
0.125	0.00583399424886412\\
0.03125	0.00291699712443206\\
0.0078125	0.00145849856221603\\
0.001953125	0.000729249281108014\\
0.00048828125	0.000364624640554007\\
};
\addlegendentry{$\ordo(\Delta t^{1/2})$}

\end{axis}
\end{tikzpicture}%
  
  \caption{The computed errors for the experiment outlined in Section~\ref{sec:results}. They are in line with the $\ordo(\Delta t^{1/4})$-prediction of Theorem~\ref{mainresult}.}
  \label{fig:main_error}
\end{figure}

\section{Conclusions}
We have proved convergence with optimal orders of a numerical scheme for an optimal control problem with both distributed and boundary control, as well as distributed and boundary Q-Wiener noise. Due to the irregularity of the noise, we can expect at most order $1/4$ in time and order $1/2$ in space. A numerical experiment confirms that this bound is optimal.

\section{Acknowledgements}
Parts of this work were completed while the second and third author were with the Max Planck Institute Magdeburg.

\bibliographystyle{plain}
\bibliography{references}

\end{document}